\documentclass[10pt]{amsart}
\usepackage{amsmath, amsfonts, amsthm, amssymb, setspace, xcolor, graphicx, soul, float, mathtools,mathrsfs,extarrows,tikz-cd,tikz,array,longtable}
\usepackage[normalem]{ulem}
\usetikzlibrary{decorations.pathreplacing}
\usepackage[a4paper,left=3.2cm,right=3.2cm,top=3cm,bottom=3cm]{geometry}
\usepackage[hidelinks]{hyperref}
\title{A Proof of Gluck's Conjecture}
\author{Baoyu Zhang}
\address{School of Mathematics, University of Birmingham, Birmingham B15 2TT, UK}
\email{baoyuzhang.math@outlook.com}
\date{August 11, 2026}
\usepackage[scr=boondox]{mathalfa}
\newtheorem{theorem}{Theorem}[section]
\newtheorem{proposition}[theorem]{Proposition}
\newtheorem{lemma}[theorem]{Lemma}
\newtheorem{corollary}[theorem]{Corollary}
\theoremstyle{definition}
\newtheorem{definition}[theorem]{Definition}

\newtheorem{conjecture}[theorem]{Conjecture}

\makeatletter
\renewcommand{\le}{\leqslant}
\renewcommand{\ge}{\geqslant}
\renewcommand{\leq}{\leqslant}
\renewcommand{\geq}{\geqslant}
\allowdisplaybreaks

\expandafter\let\expandafter\oldproof\csname\string\proof\endcsname
\let\oldendproof\endproof
\renewenvironment{proof}[1][\proofname]{%
  \oldproof[\bfseries #1]%
}{\oldendproof}
\makeatother

\newcommand{\GammaL}{\Gamma{\rm L}}
\newcommand{\pow}{\mathcal{P}}
\newcommand{\orbit}{\mathcal{O}}
\DeclareMathOperator{\Gal}{Gal}
\DeclareMathOperator{\GL}{GL}
\DeclareMathOperator{\SL}{SL}
\DeclareMathOperator{\Irr}{Irr}
\DeclareMathOperator{\Sym}{Sym}
\newcommand{\GAP}{\textsf{GAP}}
\newcommand{\Magma}{\textsf{Magma}}
\numberwithin{equation}{section}
\numberwithin{table}{section}
\begin{document}

\subjclass[2020]{Primary 20C15; Secondary 05E18, 20B15, 20D10}
\keywords{Gluck's conjecture, character degrees, Fitting subgroup,
regular orbits, solvable linear groups}

\begin{abstract}
For a finite group $H$, let $\nu(H)$ denote the maximum order of a nilpotent
subgroup of $H$. We prove that every finite solvable transitive permutation
group $P$ on a finite set $\Omega$ has a subset $\Delta\subseteq\Omega$ such
that $|P:P_\Delta|\ge\nu(P_\Delta)$. We also prove that if a finite solvable
group $H$ acts faithfully and completely reducibly on a finite module $V$,
then some $x\in V$ satisfies $|H:H_x|\ge\nu(H_x)$. Consequently, we settle
Gluck's well-known conjecture, open since 1985: every finite solvable group $G$
satisfies $|G:\mathbf{F}(G)|\le b(G)^2$, where $\mathbf{F}(G)$ is the largest
normal nilpotent subgroup of $G$ and $b(G)$ is the largest degree of an irreducible complex character of $G$.
\end{abstract}

\maketitle

\setcounter{tocdepth}{1}

\section{Introduction}

In 1985, Gluck proposed the following conjecture~\cite{Glu85}.

\begin{conjecture}[Gluck]\label{conj:gluck}
Let $G$ be a finite solvable group. Then $|G:\mathbf{F}(G)|\le b(G)^2$, where $b(G)$ is the largest degree of an irreducible complex character of $G$.
\end{conjecture}

The index $|G:\mathbf{F}(G)|$ measures how far $G$ is from being nilpotent.
$b(G)$ is a representation-theoretic invariant.
Conjecture~\ref{conj:gluck} asserts
that the two are closely tied. Our main result confirms the conjecture.

\begin{theorem}\label{thm:gluck-holds}
Conjecture~\ref{conj:gluck} is true.
\end{theorem}

Conjecture~\ref{conj:gluck} was previously established in several cases:
for groups of odd order by Espuelas~\cite{Espuelas91}, for solvable groups with abelian Sylow $2$-subgroups by Dolfi and Jabara~\cite{DolfiJabara07}, for solvable $3'$-groups by Yang~\cite{Yang11a},
for groups in which $|\mathbf{F}(G)|$ is coprime to $6$ by Cossey,
Halasi, Mar\'oti and Nguyen~\cite{CHMN15}, for solvable $S_4$-free
groups by Meng, Ballester-Bolinches and Esteban-Romero~\cite{MengBBER19}, and for a class of iterated wreath products by Meng and Guo~\cite{MengGuo25}. Here $S_4$-free means having no section
$H/K\cong S_4$, where $K\trianglelefteq H\leq G$. Related questions
involving average character degrees were studied by
Moret\'o~\cite{Moreto23}.

The right-hand side of Conjecture~\ref{conj:gluck}, the largest degree $b(G)$ of an irreducible complex character, is itself a subtle invariant. Even for symmetric groups, the largest irreducible character degrees are not fully understood and have recently been investigated computationally by Craven~\cite{Craven25}. For solvable groups, Section~\ref{sec:reduction} shows that Clifford theory connects bounds of the form $|G:\mathbf{F}(G)|\leq b(G)^c$ with estimates for their orbits on faithful completely reducible modules. Large orbits, and especially regular orbits, have therefore played an important role. See, for example, \cite{Dolfi08,KellerYang14,MengBBER20,Wolf99}. The problem of regular orbits for solvable primitive linear groups was studied by Yang and his collaborators \cite{Yang10, Yang11b, YangVV22}. We call a faithful irreducible quasi-primitive linear action of a solvable nonmetacyclic group \emph{residual} if it has no regular orbit. The classification of residual actions used here is proved by Holt and Yang~\cite{HY23}. The
companion problem of regular orbits on the {\it power set} of a permutation domain, that is, of subsets with trivial setwise stabiliser, was studied by Gluck \cite{Glu83, Glu25}, by Seress \cite{Seress97} and by Zhang \cite{Zhang97}, among others. Such results enter the imprimitive linear case through the induced permutation action on the summands.

Before the present work, the exponent in the best known general bound had been
reduced in stages: to $13/2$ by Gluck in \cite{Glu85}, to $3$ by
Moret\'o and Wolf \cite{MoretoWolf04} by way of Seress's theorem that a primitive
solvable permutation group has a base of size at most four \cite{Seress96}, and to $\log(6\cdot 24^{1/3})/\log 3\approx 2.595$ by Yang~\cite{Yang23}.

Gluck identified the obstruction to proving the bound in Conjecture~\ref{conj:gluck} in the very paper that posed the conjecture, where he wrote that his methods were ``not strong enough to give best possible bounds, even for solvable groups'' \cite[p. 443]{Glu85}. More precisely, if a solvable group $H$ acts faithfully and completely
reducibly on a finite module $V$, the approach would require a vector
$v\in V$ such that $|H_v|\le |H|^{1/2}$, equivalently $|v^H|\ge |H|^{1/2}$. But the imprimitive action of $\GL(2,2)\wr S_4$ on
$\mathbb{F}_2^{\,8}$ ``shows that this stronger version of Theorem A does
not hold''~\cite[p.~447]{Glu85}, since its largest orbit has length less than the square root of the group order. Meng and Guo~\cite[Lemma~8(2)]{MengGuo25} construct a subgroup of $\GL(4,3)$ of order $1\,152$ whose orbits on $\mathbb{F}_3^{\,4}$ have lengths $1$, $32$, $24$ and $24$. Thus no argument requiring an orbit of length at least the square root of the group order can prove the conjecture in general. We abandon the search for a single large orbit. The reduction of Section~\ref{sec:reduction} shows that Conjecture~\ref{conj:gluck} follows from the weaker assertion that $V$ always contains an \emph{abundant} element, as defined below.

\noindent{\it Notation and terminology.} We use standard notation from group
theory and character theory, as in~\cite{Isa76}. All groups and modules are
finite, and modules are written additively. For a group $H$ we write $\pi(H)$ for the set of primes dividing $|H|$, $O_p(H)$ for the largest normal $p$-subgroup of $H$, ${\rm Irr}(H)$ for the set of irreducible complex characters, $\mathbf{F}(H)$ for the Fitting subgroup, $\Phi(H)$ for the Frattini subgroup, and $\mathbf{C}_H(\cdot)$ and $\mathbf{N}_H(\cdot)$ for centralisers and normalisers, respectively. We write $\nu(H)$ for the maximum order of a nilpotent subgroup of $H$. This invariant is denoted $|H|_{\rm nil}$ by Hung and Yang in their study of bounds involving nilpotent subgroups and largest character degrees~\cite{HungYang20}. When $H$ acts on a set, $H_\Delta$ denotes the setwise stabiliser of a subset $\Delta$, and $\Delta^{c}$ its complement. For a point $x$, $H_x:=H_{\{x\}}$ is its stabiliser. Actions are written on the right, so $x^{gh}=(x^g)^h$. We set $$\rho_H(\Delta)=\frac{|H:H_\Delta|}{\nu(H_\Delta)}$$ and $\rho_H(x):=\rho_H(\{x\})$. The quantities $\rho_H(\Delta)$ and $\rho_H(x)$ are the \emph{abundance ratios} of $\Delta$ and $x$, respectively. A subset or point is \emph{abundant} if its abundance ratio is at least $1$. For any $H$-set, an $H$-orbit is \emph{regular} if the stabiliser of one (equivalently, every) point in the orbit is trivial. Accordingly, a vector $v$ is regular if $H_v=1$, and a subset $\Delta$ is regular if its setwise stabiliser $H_\Delta$ is trivial. Throughout, $q$ denotes a prime power and $\mathbb{F}_q$ a field of order $q$. A transitive permutation group is \emph{primitive} if its only invariant partitions are the partition into singletons and the partition consisting of the whole set. An irreducible linear group is {\it primitive} if it admits no nontrivial system of imprimitivity. An irreducible linear group $H\leq\GL(V)$ is {\it quasi-primitive} if every normal subgroup $N\trianglelefteq
H$ acts homogeneously on $V$, that is, if $V$ is a direct sum of mutually
isomorphic irreducible $N$-modules. This is the linear
meaning of quasi-primitivity throughout the paper. In this paper, $\Gamma$
denotes a set of coordinate indices and $\Delta$ a subset of a permutation
domain. For blocks $D_1,\ldots,D_m$ or summands $W_1,\ldots,W_m$, we write
$I=\{1,\ldots,m\}$ for the index set and regard the induced permutation group
as acting on $I$, with $D_i$ or $W_i$ represented by $i$.
Table~\ref{tab:terminology} lists the notation and terminology used repeatedly
below.
\begin{table}[H]
\caption{Where recurring notation and terminology are introduced.}
\label{tab:terminology}
\centering
\footnotesize
\renewcommand{\arraystretch}{1.25}
\begin{tabular}{l|l}
 term & location\\\hline
 \emph{residual action} & Introduction\\
 \emph{abundant} subset or point ($\rho\ge1$) & Introduction; Section~\ref{sec:reduction}\\
 \emph{regular pair} --- an ordered pair of points & Section~\ref{sec:reduction}\\
 \emph{coordinate}; \emph{coordinate set}; \emph{coordinate tuple};
 \emph{colour} & Section~\ref{sec:coordinate}\\
 \emph{local group}; \emph{top group} $P$; \emph{top factor} &
   Section~\ref{sec:coordinate}\\
 \emph{coordinate chain} $L\trianglelefteq K\trianglelefteq J$;
 \emph{bottom group} $L$ & Section~\ref{sec:coordinate}\\
 \emph{chain-abundant orbit} & Definitions~\ref{def:chain-label}
   and~\ref{def:chain-label-linear}\\
 \emph{cubic condition} & Section~\ref{sec:weighted}\\
 \emph{cubic pair} & Definition~\ref{def:cubic-pair}\\
 \emph{balanced pair} --- an ordered pair of \emph{orbits} &
   Definition~\ref{def:balanced-pair}
\end{tabular}
\end{table}

\begin{definition}\label{def:cr}
Let $H$ be a group and $V$ an $H$-module. We say that $H$ acts {\it completely
reducibly} on $V$, and that $V$ is a {\it completely reducible} $H$-module, if every
$H$-invariant subgroup of $V$ has an $H$-invariant direct complement. Equivalently,
each $O_p(V)$ is elementary abelian and completely reducible as an
$\mathbb{F}_pH$-module, and
$$V=\bigoplus_{p\in\pi(V)}O_p(V).$$
Thus every irreducible summand of $V$ is a finite-dimensional
$\mathbb{F}_pH$-module for some prime $p$.
\end{definition}

The heart of the paper is the following statement, which makes no mention of
character degrees. That it implies Conjecture~\ref{conj:gluck} is shown in
Section~\ref{sec:reduction}.

\begin{theorem}\label{thm:abundance-intro}
Let $V$ be a faithful completely reducible $H$-module for a solvable group $H$. Then $V$ contains an abundant element.
\end{theorem}

The proof of Theorem~\ref{thm:abundance-intro} uses
Theorem~\ref{thm:top-hni}, which gives an abundant subset in every solvable transitive permutation group. The proof of
Theorem~\ref{thm:top-hni} requires finite calculations only for primitive
solvable permutation groups of degree at most nine. In particular, that proof is independent of the classification of residual actions in~\cite{HY23} and of the finite calculations for primitive linear groups used later in the proof of Theorem~\ref{thm:abundance-intro}. We write $P$, rather than $H$, for the permutation group in the following theorem, since linear and permutation actions often occur together later in the paper.

\begin{theorem}\label{thm:top-hni}
Every finite solvable transitive permutation group $P$ on a finite set $\Omega$ has
an abundant subset.
\end{theorem}

\noindent{\it Organisation of the paper.} Section~\ref{sec:reduction} shows that Theorem~\ref{thm:abundance-intro} implies Conjecture~\ref{conj:gluck} and reduces the proof of Theorem~\ref{thm:abundance-intro} to faithful irreducible modules over prime fields. Section~\ref{sec:coordinate}
establishes the coordinate chain inequality, the main quantitative tool for
imprimitive modules. Sections~\ref{sec:top-hni} and~\ref{sec:weighted} prove Theorem~\ref{thm:top-hni} and its weighted strengthening for solvable
transitive permutation groups. Section~\ref{sec:primitive} proves the abundance
theorem for primitive modules, using the residual classification and finite
verification in the nonmetacyclic case and a semilinear argument in the
metacyclic case. Section~\ref{sec:local-alternatives} establishes the stronger
local alternatives used in Section~\ref{sec:imprimitive} to prove the abundance
theorem for imprimitive modules.
Section~\ref{sec:gaschutz} completes the proof of Theorem~\ref{thm:abundance-intro}
and hence of Conjecture~\ref{conj:gluck}. The appendix describes the finite
calculations used in the proof and how to repeat them.

\section{Reduction to irreducible modules}
\label{sec:reduction}

Recall that a point $x$ of an $H$-set is {\it abundant} if $\rho_H(x)\ge1$.
For an $H$-set $X$, we call $(x,y)\in X\times X$ a {\it regular pair} if
$H_x\cap H_y=1$. Equivalently, $(x,y)$ is a regular point for the diagonal
action of $H$ on $X\times X$.

\begin{lemma}\label{lem:regular-pair}
Let a finite group $H$ act on a finite set $X$, and let $(x,y)\in X\times X$
be a regular pair. Then $|H:H_x|\,|H:H_y|\ge\nu(H_x)\nu(H_y)$, and in particular
one of $x,y$ is abundant.
\end{lemma}

\begin{proof}
Since $H_x\cap H_y=1$ we have $|H_x|\,|H_y|=|H_xH_y|\le|H|$. Hence
$$\nu(H_x)\nu(H_y)\le|H_x|\,|H_y|\le|H|$$
and
$$|H:H_x|\,|H:H_y|=\frac{|H|^2}{|H_x|\,|H_y|}\ge|H|.$$
It follows that
$$\rho_H(x)\rho_H(y)=\frac{|H:H_x|\,|H:H_y|}{\nu(H_x)\nu(H_y)}\ge1,$$
and as both factors are positive, $\rho_H(x)\ge1$ or $\rho_H(y)\ge1$.
\end{proof}

The following lemma shows that Theorem~\ref{thm:abundance-intro} implies
Conjecture~\ref{conj:gluck}.

\begin{lemma}\label{lem:reduction}
If every faithful completely reducible module for a finite solvable group contains
an abundant vector, then Conjecture~\ref{conj:gluck} is true.
\end{lemma}
\begin{proof}
We choose a counterexample $G$ to Conjecture~\ref{conj:gluck} of minimal order.

We claim that $\Phi(G)=1$. Assume for a contradiction that $\Phi(G)\neq1$.
Then $G/\Phi(G)$ satisfies Conjecture~\ref{conj:gluck} by the minimal choice of
$G$. Moreover, every irreducible character of $G/\Phi(G)$ inflates to an irreducible
character of $G$, whence $b(G/\Phi(G))\le b(G)$. Since
$\mathbf{F}(G/\Phi(G))=\mathbf{F}(G)/\Phi(G)$,
$$|G:\mathbf{F}(G)|=|G/\Phi(G):\mathbf{F}(G/\Phi(G))|\le b(G/\Phi(G))^2\le b(G)^2,$$
contrary to $G$ being a counterexample to Conjecture~\ref{conj:gluck}. Hence
$\Phi(G)=1$.

Since $\Phi(K)\le\Phi(G)$ whenever $K\trianglelefteq G$, we have
$\Phi(O_p(\mathbf{F}(G)))=1$ for every $p\in\pi(\mathbf{F}(G))$. Hence every
$O_p(\mathbf{F}(G))$ is elementary abelian, and consequently $\mathbf{F}(G)$
is abelian.

Now choose $H\le G$ minimal by inclusion such that $G=\mathbf{F}(G)H$, and set
$N=H\cap\mathbf{F}(G)$. Since $\mathbf{F}(G)\trianglelefteq G$, we have
$N\trianglelefteq H$. We claim that $N\le\Phi(H)$. Otherwise some maximal
subgroup $L<H$ does not contain $N$. Then $H=NL$, and hence
$G=\mathbf{F}(G)H=\mathbf{F}(G)L$, contrary to the choice of $H$. Since
$\mathbf{F}(G)$ is abelian, it centralises
$N$. It follows from $G=\mathbf{F}(G)H$ and $N\trianglelefteq H$ that
$N\trianglelefteq G$.

Suppose that $N\neq1$. Since $\Phi(G)=1$, there is a maximal subgroup $M<G$
such that $N\not\le M$. As $N\trianglelefteq G$, maximality gives $G=NM$.
Since $N\le H$, Dedekind's modular law gives $H=H\cap NM=N(H\cap M)$.
Since $N\leq\Phi(H)$, every element of $N$ is a non-generator of $H$. The equality $H=N(H\cap M)$ therefore forces $H=H\cap M\leq M$, contrary to $N\not\leq M$. Therefore $N=1$, so $G=\mathbf{F}(G)\rtimes H$ and
$H\cong G/\mathbf{F}(G)$.

As $\mathbf{F}(G)$ is abelian, all its irreducible characters are linear.
Under pointwise multiplication, $V:=\Irr(\mathbf{F}(G))$ is its character
group. Conjugation by $H$ makes $V$ the dual of the $H$-module
$\mathbf{F}(G)$. By a theorem of Gasch\"utz~\cite[Satz~7]{Gas53}, $\Phi(G)=1$ implies that $\mathbf{F}(G)$ is completely reducible under conjugation by $G$. Since $\mathbf{F}(G)$ is abelian, it acts trivially on itself by conjugation. Thus this action induces an action of $G/\mathbf{F}(G)$ on $\mathbf{F}(G)$. It follows that $\mathbf{F}(G)$ is a completely reducible
$G/\mathbf{F}(G)$-module and hence a completely reducible $H$-module. This $H$-action is faithful because
$\mathbf{C}_G(\mathbf{F}(G))=\mathbf{F}(G)$ and
$H\cap\mathbf{F}(G)=1$. By \cite[Proposition~12.1]{ManzWolf93}, duality preserves faithfulness and irreducibility. Applying this to $\mathbf{F}(G)$ and its irreducible summands shows that $V$ is a faithful and completely reducible $H$-module. The hypothesis of the lemma
therefore gives an abundant character $\lambda\in V$ with
\begin{equation}\label{pn-stabiliser-ineq}
    |H:H_\lambda|\ge \nu(H_\lambda).
\end{equation}

Every element of
$\mathrm I_G(\lambda)=\mathbf{F}(G)\rtimes H_\lambda$
has a unique expression $fh$, where $f\in\mathbf{F}(G)$ and
$h\in H_\lambda$. Define
$$
\widehat{\lambda}\colon \mathrm I_G(\lambda)\longrightarrow\mathbb C^\times,
\qquad
\widehat{\lambda}(fh):=\lambda(f).
$$
The uniqueness of this expression makes $\widehat{\lambda}$ well defined.
For $f_1,f_2\in\mathbf{F}(G)$ and $h_1,h_2\in H_\lambda$, we have
$\lambda(h_1f_2h_1^{-1})=\lambda(f_2)$.
It follows from this equality and the linearity of $\lambda$ that the value of
$\widehat{\lambda}((f_1h_1)(f_2h_2))$ is
$$
\widehat{\lambda}((f_1h_1f_2h_1^{-1})(h_1h_2))
 =\lambda(f_1(h_1f_2h_1^{-1}))
 =\lambda(f_1)\lambda(f_2)
 =\widehat{\lambda}(f_1h_1)\widehat{\lambda}(f_2h_2).
$$
Thus $\widehat{\lambda}$ is a linear character of
$\mathrm I_G(\lambda)$ extending $\lambda$. Let $\eta\in\Irr(H_\lambda)$, and inflate $\eta$ to
$\mathrm I_G(\lambda)$ with $\mathbf{F}(G)$ in its kernel. By Gallagher's theorem \cite[Corollary~6.17]{Isa76}, $\widehat{\lambda}\eta$ is irreducible. Its induction to $G$ is irreducible by the
Clifford correspondence \cite[Theorem~6.11]{Isa76}, and has degree
$$
 |G:{\rm I}_G(\lambda)|\,(\widehat{\lambda}\eta)(1)=|G:{\rm I}_G(\lambda)|\,\eta(1)=|H:H_\lambda|\eta(1).
$$
Taking $\eta$ of largest degree gives
\begin{equation}\label{lambda-degree-bound}
    b(G)\ge |H:H_\lambda|b(H_\lambda).
\end{equation}

Since $G$ is a nontrivial solvable group, $\mathbf{F}(G)\neq1$, and hence
$|H_\lambda|\le|H|=|G:\mathbf{F}(G)|<|G|$. The group $H_\lambda$ is solvable,
so the minimality of $G$ gives
$b(H_\lambda)^2\ge |H_\lambda:\mathbf{F}(H_\lambda)|$. Squaring
\eqref{lambda-degree-bound} and using this gives
$$b(G)^2
\ge |H:H_\lambda|^2\,b(H_\lambda)^2
\ge |H:H_\lambda|^2\,|H_\lambda:\mathbf{F}(H_\lambda)|
=\frac{|H:H_\lambda|}{|\mathbf{F}(H_\lambda)|}\,|H|.
$$
Since $\mathbf{F}(H_\lambda)$ is nilpotent,
$|\mathbf{F}(H_\lambda)|\le\nu(H_\lambda)$, while
\eqref{pn-stabiliser-ineq} gives
$|H:H_\lambda|\ge\nu(H_\lambda)\ge|\mathbf{F}(H_\lambda)|$. Hence
$b(G)^2\ge |H|=|G:\mathbf{F}(G)|$, contradicting the choice of $G$.
\end{proof}

The reduction to irreducible modules is by induction, removing one irreducible
summand at each step. The following lemma is the inductive step. The key point is that the
ratio for the sum is bounded below by the product of the two ratios obtained
from the summands.

\begin{lemma}\label{lem:direct-sum}
Let $H$ act faithfully and completely reducibly on $V=W\oplus U$, where $W$
and $U$ are nonzero $H$-submodules. Set $A=\mathbf{C}_H(W)$ and
$\overline{H}=H/A$. If $W$ contains an abundant vector for $\overline{H}$ and
$U$ contains an abundant vector for $A$, then $V$ contains an abundant vector
for $H$.
\end{lemma}

\begin{proof}
The quotient $\overline{H}$ acts faithfully on $W$ by
the definition of $A$, and completely reducibly because $W$ is an $H$-summand. The
subgroup $A$ is normal in $H$. Applying Clifford's theorem
\cite[Theorem~0.1]{ManzWolf93} to each irreducible $H$-summand of $V$ shows that
the restriction of $V$ to $A$ is completely reducible, and hence so is its summand
$U$. The action of $A$ on $U$ is faithful: if an element of $A$ acts trivially
on $U$, then it acts trivially on both $W$ and $U$, and hence on $V$.

Fix $z=w+u$ with $w\in W$ and $u\in U$. Since the summands are $H$-invariant,
$H_z=H_w\cap H_u$. Since $A\le H_w$, we have
$H_zA/A\le(\overline H)_w$ and
$A\cap H_z=A_u$, whence
\begin{equation}\label{eq:ds-index}
  |H:H_z|=|H:H_zA|\,|H_zA:H_z|\ge |\overline{H}:(\overline H)_w|\,|A:A_u|.
\end{equation}
If $N\le H_z$ is nilpotent, then both its image
$NA/A\le(\overline H)_w$ and its intersection
$N\cap A\le A_u$ are nilpotent. Moreover,
$|N|=|NA/A|\,|N\cap A|$. Therefore
$|N|\le\nu((\overline H)_w)\,\nu(A_u)$, and maximising over $N$ yields
\begin{equation}\label{eq:ds-nu}
  \nu(H_z)\le \nu((\overline H)_w)\,\nu(A_u).
\end{equation}
Now choose $w\in W$ with $\rho_{\overline{H}}(w)\ge1$ and $u\in U$ with $\rho_A(u)\ge1$,
and set $z=w+u$. Combining \eqref{eq:ds-index} and~\eqref{eq:ds-nu} gives
$\rho_H(z)\ge\rho_{\overline{H}}(w)\,\rho_A(u)\ge1$, so $z$ is abundant.
\end{proof}

\begin{corollary}\label{cor:reduce-irreducible}
If Theorem~\ref{thm:abundance-intro} fails, it fails for some faithful
irreducible $\mathbb{F}_pH$-module.
\end{corollary}

\begin{proof}
The zero module is not a counterexample: if it is faithful, then $H=1$, and
its only vector is abundant. Take a counterexample with $|V|$ minimal. By
complete reducibility, $V$ is a direct sum
of irreducible elementary abelian submodules. If there is more than one summand,
write $V=W\oplus U$ with $W$ a single summand and $U$ the sum of the others.
Both $W$ and $U$ are smaller than $V$. As shown in the proof of
Lemma~\ref{lem:direct-sum}, $H/\mathbf{C}_H(W)$ acts faithfully and completely
reducibly on $W$, while $\mathbf{C}_H(W)$ acts faithfully and completely
reducibly on $U$. Minimality therefore gives an abundant vector for each
action, and Lemma~\ref{lem:direct-sum} gives one for $V$, a contradiction.
Hence $V$ is itself irreducible and elementary abelian, that is, a
finite-dimensional vector space over a single prime field.
\end{proof}

It remains to find an abundant vector in a faithful
irreducible $\mathbb{F}_pH$-module. We handle the primitive and imprimitive cases in
Sections~\ref{sec:primitive} and~\ref{sec:imprimitive}, after developing the
tools they need.

\section{Coordinate chains in imprimitive actions}
\label{sec:coordinate}

The estimate in this section will be used in two settings. If
$V=W_1\oplus\cdots\oplus W_m$, then every vector $v\in V$ is uniquely
determined by its components $v_i\in W_i$. If
$\Omega=D_1\mathbin{\dot\cup}\cdots\mathbin{\dot\cup}D_m$, then every subset
$\Delta\subseteq\Omega$ is uniquely determined by the subsets
$\Delta\cap D_i$. Thus in both settings the object whose stabiliser we study
is represented by a tuple, while the group permutes the summands or parts and
acts within them. We now formulate the estimate for abstract sets $X_i$. In
the applications at the end of the section, $X_i$ will be $\{i\}\times W_i$
in the linear case and $\{i\}\times\pow(D_i)$ in the permutation case, where
$\pow(D_i)$ denotes the power set of $D_i$.

Let $\mathcal X=X_1\mathbin{\dot\cup}\cdots\mathbin{\dot\cup}X_m$ be a
finite set partitioned into nonempty parts, let $I=\{1,\ldots,m\}$, and
let $H\leq\Sym(\mathcal X)$ permute the parts transitively. Every $h\in H$
therefore induces a permutation of $I$. We write $j=i^h$ when
$X_i^h=X_j$. We call $i$ a {\it coordinate} and $X_i$ its
{\it coordinate set}.

The group $H$ also acts on the Cartesian product $\prod_{i\in I}X_i$. If
$x=(x_i)_{i\in I}\in\prod_{i\in I}X_i$, we call $x$ a {\it coordinate
tuple}. For $h\in H$ and $j=i^h$, the same element $h$ sends
$x_i\in X_i$ to $x_i^h\in X_j$, and its action on coordinate tuples is
defined by $(x^h)_j=x_i^h$.
The disjoint union $\mathcal X$ carries the given action of $H$, whereas the
object whose stabiliser we study is the tuple $x\in\prod_iX_i$.

Here is a small example of the two actions. Take
$$
  X_1=\{0_1,1_1,2_1\},\qquad X_2=\{0_2,1_2,2_2\},
$$
and let $H=\langle d,\tau\rangle$, where
$$
  d=(1_1\ 2_1)(1_2\ 2_2),\qquad
  \tau=(0_1\ 0_2)(1_1\ 1_2)(2_1\ 2_2).
$$
Thus $d$ acts inside the two coordinate sets and $\tau$ interchanges them.
The disjoint union has six points, whereas $X_1\times X_2$ has nine tuples.
For instance,
$$
  (0_1,1_2)^d=(0_1^d,1_2^d)=(0_1,2_2),\qquad
  (0_1,1_2)^\tau=(1_2^\tau,0_1^\tau)=(1_1,0_2).
$$
In the calculation with $\tau$, the entries of the image tuple are written in
the order $X_1,X_2$: the entry $1_2$ is sent into $X_1$, while $0_1$ is sent
into $X_2$.

Let $B$ be the kernel of the action of $H$ on $I$, and write
$P=H/B\leq\Sym(I)$ for the permutation group induced by $H$ on $I$.
Thus $B$ consists of the elements of $H$ that stabilise every coordinate
set setwise. We call $P$ the {\it top group}.

Let $M_i$ be the setwise stabiliser of $X_i$ in $H$, and write
$$
  r_i:M_i\longrightarrow\Sym(X_i),\qquad r_i(h)=h|_{X_i}
$$
for the restriction homomorphism. Define
$$
  J_i=r_i(M_i),\qquad K_i=r_i(B).
$$
We call $J_i$ the {\it local group} on $X_i$. Since $B\trianglelefteq H$, we
have $K_i\trianglelefteq J_i$. Although $H$ is transitive on the family
$\{X_1,\ldots,X_m\}$, the group $J_i$ need not be transitive on $X_i$.

Choose $t_i\in H$ with $X_1^{t_i}=X_i$, where $t_1=1$. The restriction of
$t_i$ is a bijection from $X_1$ to $X_i$. These bijections identify all the
coordinate sets with $X_1$. They also identify the setwise stabiliser in
$\Sym(\mathcal X)$ of the partition $\{X_1,\ldots,X_m\}$ with the wreath
product $\Sym(X_1)\wr S_m$ in its usual imprimitive action on $\mathcal X$.
Hence $H$ is a subgroup of this wreath product. The
subgroup $\Sym(X_1)^m$ consists of the permutations that stabilise each
coordinate set setwise, and hence $B=H\cap\Sym(X_1)^m$.
This does not mean that the restrictions of elements of $B$ to the individual
coordinate sets can be chosen independently. In the example, the two
transpositions in $d$ always occur together.

In the example above,
$$
  B=\langle d\rangle,\qquad P=\langle\tau B\rangle\cong C_2,
  \qquad J_1=\langle(1_1\ 2_1)\rangle,
$$
and the $H$-orbits on $\mathcal X$ are
$$
  \mathcal O_0=\{0_1,0_2\}\qquad\text{and}\qquad
  \mathcal O_1=\{1_1,2_1,1_2,2_2\}.
$$
For a coordinate tuple $x=(x_i)_{i\in I}$, define a map $c_x$ on $I$ by
$c_x(i)=x_i^H$. Thus $c_x(i)$ is the $H$-orbit on $\mathcal X$ containing
$x_i$. We call this orbit the {\it colour} of coordinate $i$ in $x$. For a
related use of inequivalent local colourings as colours in an imprimitive action, see
\cite[Section~2.3]{SabatiniStabilizers}. For
$x=(0_1,1_2)$, we have $c_x(1)=\mathcal O_0$ and
$c_x(2)=\mathcal O_1$, so the two coordinates have different colours. For
$y=(1_1,2_2)$, both coordinates have colour $\mathcal O_1$.

The colour may also be read from the local action on $X_1$. The chosen
element $t_i$ identifies $X_i$ with $X_1$ by sending $u\in X_i$ to
$u^{t_i^{-1}}$. Moreover,
$$
  x_i^H\cap X_1=\bigl(x_i^{t_i^{-1}}\bigr)^{J_1}.
$$
Indeed, an element of $H$ carries $X_i$ to $X_1$ precisely when it has the
form $t_i^{-1}m$ for some $m\in M_1$. In particular, the $J_1$-orbit on the
right is independent of the choice of $t_i$. Every $H$-orbit on
$\mathcal X$ meets $X_1$. Consequently, coordinates $i$ and $j$ have the
same colour precisely when
$$
  c_x(i)=c_x(j)
  \quad\Longleftrightarrow\quad
  \bigl(x_i^{t_i^{-1}}\bigr)^{J_1}
  =\bigl(x_j^{t_j^{-1}}\bigr)^{J_1}.
$$
In the example, take $t_2=\tau$. Then $0_1$ and
$1_2^{t_2^{-1}}=1_1$ belong to the distinct
$J_1$-orbits $\{0_1\}$ and $\{1_1,2_1\}$.

Since $P$ acts on $I$, define the stabiliser in $P$ of this colouring by
$$
  T_x=\{g\in P:c_x(i^g)=c_x(i)\text{ for every }i\in I\}.
$$
An element of $P$ lies in $T_x$ precisely when it moves each index only to an
index carrying the same colour. We call $|P:T_x|/\nu(T_x)$ the {\it top
factor} associated with $x$. In the example, the nonidentity element of $P$
interchanges the two indices. Consequently $T_x=1$ for $x=(0_1,1_2)$, while
$T_y=P$ for $y=(1_1,2_2)$.

To define the local factors, fix the order $1,\ldots,m$ of the coordinate
sets and define
$$
  B_0=B,\qquad
  B_i=B_{i-1}\cap\ker r_i,\qquad
  L_i=r_i(B_{i-1})\qquad(1\leq i\leq m).
$$
Thus $B_i$ is the pointwise stabiliser in $B$ of
$X_1\cup\cdots\cup X_i$, while $L_i$ is the action on $X_i$ induced by the elements
of $B$ that act trivially on $X_1,\ldots,X_{i-1}$. Since $B\leq M_i$ and
$\ker r_i\trianglelefteq M_i$, induction gives $B_i\trianglelefteq B$. Hence
$$
  L_i\trianglelefteq K_i\trianglelefteq J_i.
$$
We call this the {\it coordinate chain} at $i$, and $L_i$ its
{\it bottom group}. Only $L_i$ occurs in the inequality below. We retain
$K_i$ and $J_i$ because the later arguments and finite calculations consider
every chain $L\trianglelefteq K\trianglelefteq J$. The groups $B_i$ and
$L_i$ may depend on the chosen order. For every order, however,
$$
  \prod_{i=1}^m|L_i|=|B|.
$$
Indeed, $|L_i|=|B_{i-1}:B_i|$, and the faithful action of $H$ on
$\mathcal X$ gives $B_m=1$.

The same example shows why these chains are needed. The element $d$ acts
nontrivially on both $X_1$ and $X_2$, but its action on either set determines
its action on the other. With the order $1,2$, we have
$$
  B_0=B,\qquad B_1=1,\qquad L_1\cong C_2,\qquad L_2=1.
$$
After the action on $X_1$ has been taken into account, no nonidentity element
of $B$ remains to act on $X_2$. For comparison, if
$$
  B=\langle(1_1\ 2_1),(1_2\ 2_2)\rangle,
$$
then the actions on $X_1$ and $X_2$ are independent, and
$L_1\cong L_2\cong C_2$.

We shall also use the following elementary inequality. If
$S\leq M\leq L$, then $\nu(S)\leq\nu(M)$ and
\begin{equation}\label{eq:top-monotone}
  \frac{|L:S|}{\nu(S)}
  \ \geq\
  \frac{|L:M|}{\nu(M)}.
\end{equation}

\begin{theorem}[Coordinate chain inequality]\label{thm:coordinate}
For the coordinate chains $L_i\trianglelefteq K_i\trianglelefteq J_i$
associated with the chosen order, every coordinate tuple
$x=(x_i)_{i\in I}$ satisfies
\begin{equation}\label{eq:coord-clean}
  \rho_H(x)\ \geq\
  \frac{|P:T_x|}{\nu(T_x)}
  \prod_{i=1}^m\rho_{L_i}(x_i).
\end{equation}
\end{theorem}

\begin{proof}
We first separate the action of the stabiliser on the coordinate sets from its
intersection with $B$. Define
$$
  B_x=B\cap H_x,
  \qquad
  R=H_xB/B\leq P.
$$
Then $|H:H_x|=|P:R|\,|B:B_x|$.
If $N\leq H_x$ is nilpotent, then the homomorphism from $N$ to $R$ has
kernel $N\cap B=N\cap B_x$. Its image and kernel are nilpotent, and hence
$|N|\leq\nu(R)\nu(B_x)$. Therefore
\begin{equation}\label{eq:top-base-split}
  \rho_H(x)\ \geq\
  \frac{|P:R|}{\nu(R)}
  \frac{|B:B_x|}{\nu(B_x)}.
\end{equation}

We next bound the top contribution in~\eqref{eq:top-base-split}. Suppose
$h\in H_x$ carries $X_i$ to $X_j$. Since $x^h=x$, the
element $h$ also carries $x_i$ to $x_j$. These two entries therefore have
the same colour. It follows that the image $R$ of $H_x$ in $P$ preserves
the map $c_x$, so $R\leq T_x$. Applying~\eqref{eq:top-monotone} gives
\begin{equation}\label{eq:top-bound}
  \frac{|P:R|}{\nu(R)}
  \ \geq\
  \frac{|P:T_x|}{\nu(T_x)}.
\end{equation}

It remains to bound the contribution from $B$. Recall that
$B_x=B\cap H_x$. If $b\in B_x$, then $i^b=i$ because $b\in B$, while
$x^b=x$ because $b\in H_x$. Hence $x_i^b=(x^b)_i=x_i$ for every $i$.
Thus $B_x$ fixes every entry of $x$. This need not be true of the whole
stabiliser $H_x$: in the example above,
$\tau\in H_{(0_1,0_2)}\setminus B$ fixes the tuple $(0_1,0_2)$, although
$0_1^\tau=0_2$ and $0_2^\tau=0_1$.

For $1\leq i\leq m$, define $S_i=r_i(B_x\cap B_{i-1})$.
Since $B_x\cap B_{i-1}\leq B_{i-1}$, we have $S_i\leq L_i$. If $s\in S_i$,
then $s=r_i(b)$ for some $b\in B_x\cap B_{i-1}$, and the preceding paragraph
gives $x_i^s=x_i^b=x_i$. Hence $S_i\leq(L_i)_{x_i}$. This inclusion may be
strict: an element of $L_i$ fixing $x_i$ need not have a preimage in
$B_{i-1}$ that fixes $x_{i+1},\ldots,x_m$.

Recall that $B_m=1$. The normality of
the groups $B_i$ in $B$ gives a chain
$$
  B=B_xB_0\geq B_xB_1\geq\cdots\geq B_xB_m=B_x.
$$
Multiplying the successive indices in this chain gives
$$
  |B:B_x|=
  \prod_{i=1}^m|B_xB_{i-1}:B_xB_i|.
$$
For each $i$, the product formula gives the first
equality below. Since $B_i\leq B_{i-1}$, Dedekind's modular law gives the
second:
$$
  |B_xB_{i-1}:B_xB_i|
  =|B_{i-1}:B_{i-1}\cap(B_xB_i)|
  =|B_{i-1}:(B_x\cap B_{i-1})B_i|.
$$
By definition, $r_i(B_{i-1})=L_i$, and the kernel of $r_i$ on $B_{i-1}$ is
$B_{i-1}\cap\ker r_i=B_i$. Since $r_i(B_i)=1$, the restriction
$\left.r_i\right|_{B_{i-1}}$ maps $(B_x\cap B_{i-1})B_i$ onto $S_i$. A
surjective homomorphism preserves the index of a subgroup containing its
kernel, so
$|B_{i-1}:(B_x\cap B_{i-1})B_i|=|L_i:S_i|$.
Multiplying these equalities over all $i$ gives
\begin{equation}\label{eq:base-index}
  |B:B_x|=\prod_{i=1}^m|L_i:S_i|.
\end{equation}

Let $N\leq B_x$ be nilpotent. Since $B_i\trianglelefteq B$, the groups
$N\cap B_i$ form a normal series from $N$ to $1$.
For each $i$, the kernel of
$\left.r_i\right|_{N\cap B_{i-1}}$ is $N\cap B_i$. The first isomorphism
theorem therefore gives
$$
  \frac{N\cap B_{i-1}}{N\cap B_i}
  \cong r_i(N\cap B_{i-1})\leq S_i.
$$
The image is nilpotent, so each factor has order at most $\nu(S_i)$, and hence
$$
  |N|
  \leq\prod_{i=1}^m\nu(S_i).
$$
Maximising over all nilpotent subgroups $N$ of $B_x$ yields
\begin{equation}\label{eq:base-nu}
  \nu(B_x)\leq\prod_{i=1}^m\nu(S_i).
\end{equation}
Combining~\eqref{eq:base-index} and~\eqref{eq:base-nu}, and then applying
\eqref{eq:top-monotone} with
$S=S_i$, $M=(L_i)_{x_i}$ and $L=L_i$, gives
\begin{equation}\label{eq:base-bound}
  \frac{|B:B_x|}{\nu(B_x)}
  \ \geq\
  \prod_{i=1}^m\frac{|L_i:S_i|}{\nu(S_i)}
  \ \geq\
  \prod_{i=1}^m\rho_{L_i}(x_i).
\end{equation}
The theorem now follows from
\eqref{eq:top-base-split},~\eqref{eq:top-bound} and~\eqref{eq:base-bound}.
\end{proof}

We now apply Theorem~\ref{thm:coordinate} in the two settings described at
the beginning of the section. The tags below make the coordinate sets
disjoint.

For the linear application, suppose that $H$ acts faithfully and
imprimitively on $V=W_1\oplus\cdots\oplus W_m$, permuting the summands
transitively. The
summands are not disjoint as sets because they all contain the zero vector.
We therefore use the tagged copies $X_i=\{i\}\times W_i$. The tag $i$ serves
only to distinguish these copies. The action on them is
$(i,v)^h=(i^h,v^h)$.
The map
$$
  \prod_{i\in I}X_i\longrightarrow V,
  \qquad
  \bigl((i,v_i)\bigr)_{i\in I}\longmapsto
  \sum_{i\in I}v_i
$$
is a bijection because the sum is direct. If
$x=\bigl((i,v_i)\bigr)_{i\in I}$ and $j=i^h$, then the $j$th entry of $x^h$
is $(j,v_i^h)$. Hence the vector corresponding to $x^h$ is
$$
  \sum_{i\in I}v_i^h=\left(\sum_{i\in I}v_i\right)^h,
$$
so the bijection is $H$-equivariant. Thus the coordinate tuple corresponds to the
coordinate decomposition of a vector, and the stabiliser of the tuple is the
stabiliser of that vector. The tagged action is faithful, and we identify the
local permutation group on $X_i$ with the corresponding group induced on
$W_i$.

The action on $\mathcal X$ is not transitive: the tagged zero vectors form
one orbit, while the nonzero vectors need not form a single orbit. Since
Theorem~\ref{thm:top-hni} concerns transitive actions, it cannot be applied
directly to $\mathcal X$.

Nor would an abundant subset $\Delta\subseteq\mathcal X$ by itself give an
abundant vector. Such a subset may contain several elements of one coordinate
set $X_i$ and none of another, whereas a coordinate tuple contains exactly
one element of each $X_i$. The map
$$
  \pow(\mathcal X)\longrightarrow V,\qquad
  \Delta\longmapsto\sum_{(i,w)\in\Delta}w
$$
is $H$-equivariant. If this map sends $\Delta$ to $v$, then
$H_\Delta\leq H_v$, and this inclusion may be strict. It follows
from~\eqref{eq:top-monotone} that $\rho_H(\Delta)\geq\rho_H(v)$. Thus the
abundance of $\Delta$ does not imply that $v$ is abundant.

Section~\ref{sec:imprimitive} proves the imprimitive linear case by applying
Theorem~\ref{thm:top-hni}, or its weighted strengthening
Theorem~\ref{thm:weighted}, to the top group $P=H/B$ acting on $I$.
Theorem~\ref{thm:coordinate} then combines the resulting top estimate with
the local ratios $\rho_{L_i}(x_i)$ arising from the bottom groups $L_i$.

For the application to permutation groups, suppose that $H$ acts faithfully on
$\Omega=D_1\mathbin{\dot\cup}\cdots\mathbin{\dot\cup}D_m$ and permutes the
sets $D_i$ transitively. Take
$X_i=\{i\}\times\pow(D_i)$ with the natural action. The map
$$
  \prod_{i\in I}X_i\longrightarrow\pow(\Omega),
  \qquad
  \bigl((i,\Delta_i)\bigr)_{i\in I}\longmapsto
  \bigcup_{i\in I}\Delta_i
$$
is a bijection because the sets $D_i$ are disjoint. If
$x=\bigl((i,\Delta_i)\bigr)_{i\in I}$, then $h$ carries $\Delta_i$ to
$\Delta_i^h\subseteq D_{i^h}$, and the subset corresponding to $x^h$ is
$$
  \bigcup_{i\in I}\Delta_i^h
  =\left(\bigcup_{i\in I}\Delta_i\right)^h.
$$
Thus the bijection is $H$-equivariant, and the stabiliser of the coordinate
tuple is the setwise stabiliser of the corresponding subset. For a subset
$\Delta\subseteq\Omega$, the local subset at $i$ is
$\Delta_i=\Delta\cap D_i$. The tagged action is faithful, since its kernel
fixes every singleton of $\Omega$. The groups induced by $M_i$ on
$\pow(D_i)$ and $D_i$ are naturally isomorphic, and we identify them.
The bijections $X_1\longrightarrow X_i$ induced by the elements $t_i$
identify the local groups and their orbits. When a local orbit on $X_1$ is
used on $X_i$, we denote its image by the same symbol.

\section{Abundant subsets of solvable transitive groups}
\label{sec:top-hni}

We now prove the first of the two results about permutation groups promised in the
introduction: a solvable transitive group always has an abundant subset.

The induction uses an {\it unrefinable} block system, by which we mean a
nontrivial block system with no strictly finer nontrivial block system. Its local group is primitive, since a
system of imprimitivity for the local action would transport to a proper
refinement of the original block system. Here the {\it local group} is the permutation group induced on a block by
its setwise stabiliser, and the quotient $\overline{P}$ is the permutation group induced by $P$ on the set of blocks.

\begin{definition}\label{def:chain-label}
Let $D$ be a finite set and let $J\leq\Sym(D)$ be a primitive permutation
group. A $J$-orbit $\orbit$ on the power set $\pow(D)$ is a
{\it chain-abundant orbit} if, for every chain of subgroups
$L\trianglelefteq K\trianglelefteq J$, there exists
$\Delta\in\orbit$ with $\rho_L(\Delta)\geq1$. The same orbit $\orbit$ is
used for every chain, but the representative chosen from it may vary.
\end{definition}

Recall that the degree of a finite solvable primitive permutation group is a prime power. The possible nontrivial degrees at most nine are therefore $2,3,4,5,7,8$ and $9$. We deal with the smallest degrees here and use a short computation for the rest later.

\begin{lemma}\label{lem:deg34}
Let $J$ be a solvable primitive permutation group on a finite set $D$ with
$|D|=3$ or $4$. Then the singletons in $D$ and the subsets of $D$ of size $2$ form two
distinct chain-abundant $J$-orbits on $\pow(D)$.
\end{lemma}

\begin{proof}
Fix a chain $L\trianglelefteq K\trianglelefteq J$. If $|D|=3$, then the primitive
groups are $C_3$ and $S_3$, and $L$ is one of $1,C_3,S_3$. For $L=1$ both
ratios are $1$. For $L=C_3$, both a singleton and a subset of size $2$ have
trivial stabiliser, so their abundance ratios are $3$. For $L=S_3$, both
stabilisers have order $2$, so the abundance ratios are $3/2$.

If $|D|=4$, then the primitive groups are $A_4$ and $S_4$, and the possible
bottom groups are
$1,C_2,V_4,A_4,S_4$. For each $L$, the largest ratios
among the singletons and among the subsets of size $2$ are
$$
\renewcommand{\arraystretch}{1.15}
\begin{array}{c|ccccc}
 L & 1 & C_2 & V_4 & A_4 & S_4\\\hline
 \text{singleton} & 1 & 2 & 4 & 4/3 & 4/3\\
 \text{size }2    & 1 & 2 & 1 & 3 & 3/2
\end{array}
$$
Here each maximum is taken over the relevant $J$-orbit. For example, with
$L=\langle(12)(34)\rangle\cong C_2$ the subset $\{1,3\}$ has trivial
stabiliser in $L$ and ratio $2$, whereas $\{1,2\}$ gives only $1/2$. One checks the
entries from the orbit sizes and from the orders of the relevant setwise stabilisers,
$\nu(C_2)=2$, $\nu(C_3)=\nu(S_3)=3$, $\nu(V_4)=\nu(A_4)=4$ and
$\nu(S_4)=8$. Every entry is at least $1$.
\end{proof}

\begin{lemma}\label{lem:deg5789}
Let $J$ be a solvable primitive permutation group on a finite set $D$ with
$|D|\in\{5,7,8,9\}$. Then $J$ has at least two distinct chain-abundant
orbits on $\pow(D)$. The two orbits may be chosen on subsets of sizes $1$ and
$2$, except when $J$ is $\mathrm{A}\Gamma\mathrm{L}(1,9)$ or $\mathrm{AGL}(2,3)$
in degree nine. For these two groups, subsets of sizes $2$ and $3$ suffice.
\end{lemma}

\begin{proof}
Among the $21$ solvable primitive groups of degree at most nine in the
primitive groups library of \GAP~\cite{GAP4}, exactly $16$ have degree in
$\{5,7,8,9\}$. Up to relabelling the elements of $D$, the group $J$ is one
of these $16$ groups. Relabelling does not change whether an orbit is
chain-abundant. The calculation in Appendix~\ref{sec:comp-perm} checks all
$21$ groups and hence, in particular, the $16$ cases relevant here. For each
relevant group $J$, it considers every subgroup $L$ that
can occur as the bottom group of a chain
$L\trianglelefteq K\trianglelefteq J$ and every $J$-orbit $\orbit$ on
$\pow(D)$. It verifies that two orbits of the sizes
stated in the lemma contain, for every such $L$, a subset $\Delta$ satisfying
$\rho_L(\Delta)\geq1$. For the two exceptional groups the point
stabilisers $\GammaL(1,9)$ and $\GL(2,3)$ contain nilpotent subgroups of order
$16$. Thus every singleton $\Delta$ satisfies $\rho_J(\Delta)\le9/16<1$ for the
chain $L=K=J$, so no singleton orbit is chain-abundant. For each of these two
groups, Appendix~\ref{sec:comp-perm} finds two chain-abundant $J$-orbits, one
consisting of subsets of size $2$ and one consisting of subsets of size $3$.
\end{proof}

For larger degrees we use Gluck's theorem on regular subsets in place of any
computation.

\begin{lemma}\label{lem:large}
Let $J$ be a solvable primitive permutation group on a finite set $D$ with
$|D|>9$. Then there exists a regular subset $\Delta_0\subseteq D$ with
$|\Delta_0|\ne|D|/2$. In particular, the $J$-orbits of $\Delta_0$ and
$\Delta_0^c$ are distinct and chain-abundant.
\end{lemma}

\begin{proof}
By~\cite[Theorem~1]{Glu83} there is a regular subset $\Delta_0$ with
$|\Delta_0|\ne|D|/2$. Its complement $\Delta_0^{c}$ also has
trivial stabiliser. The subsets $\Delta_0$ and $\Delta_0^c$ have different
sizes and therefore lie in distinct $J$-orbits. For every chain
$L\trianglelefteq K\trianglelefteq J$, both $L_{\Delta_0}$ and
$L_{\Delta_0^c}$ are trivial. Hence
$\rho_L(\Delta_0)=\rho_L(\Delta_0^c)=|L|\ge1$.
\end{proof}

The next observation concerns a coordinate tuple with exactly two colours.

\begin{lemma}\label{lem:fine-label}
Let $x=(x_i)_{i\in I}$ be a coordinate tuple with exactly two colours. Let
$\mathcal O$ be one of these colours and set
$\Gamma=\{i\in I:c_x(i)=\mathcal O\}$. Then $T_x\leq P_\Gamma$.
\end{lemma}

\begin{proof}
Let $g\in T_x$. If $i\in\Gamma$, then
$c_x(i^g)=c_x(i)=\mathcal O$, so $i^g\in\Gamma$. Thus $g$ stabilises
$\Gamma$ setwise and hence belongs to $P_\Gamma$.
\end{proof}

The case where the blocks $D_i$ have size two needs a different argument.

\begin{lemma}\label{lem:c2-balance}
Let $I=\{1,\ldots,m\}$. For each $i\in I$, let $D_i$ be a set of size $2$,
let $L_i\leq\Sym(D_i)$, and choose a singleton $E_i\subseteq D_i$. Fix
$\Gamma\subseteq I$, and define
$$
 \Delta_i=
 \begin{cases}
  E_i,&i\in\Gamma,\\
  \varnothing,&i\notin\Gamma,
 \end{cases}
 \qquad
 \Delta_i'=
 \begin{cases}
  \varnothing,&i\in\Gamma,\\
  E_i,&i\notin\Gamma.
 \end{cases}
$$
Then
$$
 \left(\prod_{i\in I}\rho_{L_i}(\Delta_i)\right)
 \left(\prod_{i\in I}\rho_{L_i}(\Delta_i')\right)=1.
$$
Consequently at least one of the two products is at least $1$.
\end{lemma}

\begin{proof}
Since $D_i$ has two elements, either $L_i=1$ or
$L_i=\Sym(D_i)\cong C_2$. If $L_i=1$, both local abundance ratios are $1$.
If $L_i\cong C_2$, the singleton has abundance ratio $2$, while the empty
set has abundance ratio $1/2$. Thus
$$
 \rho_{L_i}(\Delta_i)\rho_{L_i}(\Delta_i')=1
$$
for every $i\in I$. Multiplying these equalities gives the displayed
identity. Since both products are positive, at least one of them is at
least $1$.
\end{proof}

\begin{proof}[Proof of Theorem~\textup{\ref{thm:top-hni}}]
Induct on $|\Omega|$, the case $|\Omega|=1$ being trivial. Suppose first that $P$
is primitive. If $|\Omega|>9$, then a regular subset supplied by
Lemma~\ref{lem:large} is abundant. Degrees three and four are covered by Lemma~\ref{lem:deg34}, degrees
$5,7,8,9$ by Lemma~\ref{lem:deg5789}, and in degree two a singleton has trivial
stabiliser. These are all primitive cases.

Now let $P$ be imprimitive, and fix an unrefinable block system
$D_1,\ldots,D_m$. Let $I=\{1,\ldots,m\}$, where $i$ indexes $D_i$, and let
$J$ be the local primitive group on $D_1$. Let $\overline P\leq\Sym(I)$ be
the group induced by $P$ on $I$. By induction choose
$\Gamma\subseteq I$ with
$\rho_{\overline{P}}(\Gamma)\ge1$. Since $\overline{P}$ is nontrivial,
$\rho_{\overline{P}}(\varnothing)=\rho_{\overline{P}}(I)=1/\nu(\overline{P})<1$.
The abundant subset $\Gamma$ is therefore neither empty nor all of $I$. Thus
both $\Gamma$ and $I\setminus\Gamma$ are nonempty.

If $J\not\cong C_2$, the construction below adapts the proof of
\cite[Theorem~2]{Glu83}, with an abundant top subset and two
chain-abundant local orbits in place of the corresponding regular subsets.
The group $J$ has at least two distinct chain-abundant orbits by
Lemmas~\ref{lem:deg34},~\ref{lem:deg5789} and~\ref{lem:large}. For each
$i\in I$, choose $\Delta_i\subseteq D_i$ in the first orbit if
$i\in\Gamma$, and in the second orbit otherwise, so that
$\rho_{L_i}(\Delta_i)\geq1$. Let $\Delta$ be the union of the subsets
$\Delta_i$. Let $T\leq\overline P$ be
the subgroup preserving the colours of the corresponding coordinate tuple.
By Lemma~\ref{lem:fine-label}, $T\leq\overline{P}_\Gamma$. Hence
$|\overline{P}:T|/\nu(T)\ge\rho_{\overline{P}}(\Gamma)\ge1$. Every local
factor in~\eqref{eq:coord-clean} is also at least $1$. It follows that
$\rho_P(\Delta)\ge1$.

If $J\cong C_2$, choose a singleton $E_i\subseteq D_i$ for every $i\in I$.
Use $E_i$ when $i\in\Gamma$ and the empty subset of $D_i$ otherwise. The two
local orbits are distinct. If $T\leq\overline P$ is the subgroup preserving
the colours of the corresponding coordinate tuple, then
$T\leq\overline{P}_\Gamma=\overline{P}_{\Gamma^c}$, and the top factor is at
least $1$. By Lemma~\ref{lem:c2-balance}, either this choice of local subsets
or the choice obtained by interchanging the singleton and the empty set on
every block has product of local ratios at least $1$. Let $T'$ be the subgroup
preserving the colours of the interchanged tuple. Its singleton coordinates
are indexed by $\Gamma^c$, so
$T'\leq\overline P_{\Gamma^c}=\overline P_\Gamma$. Thus the top factor is at
least $1$ for either tuple. Let $\Delta$ be the union of the local subsets in
whichever of the two choices has product of local ratios at least $1$.
Applying Theorem~\ref{thm:coordinate} to the corresponding coordinate tuple
gives $\rho_P(\Delta)\ge1$.
\end{proof}

\section{A weighted transitive theorem}
\label{sec:weighted}

Theorem~\ref{thm:weighted} is a weighted version of
Theorem~\ref{thm:top-hni}. Each point $i\in\Omega$ carries two positive real
weights $z_i$ and $n_i$, with the only condition that $z_i n_i^3\ge1$. We
call this inequality the \emph{cubic condition}. For a subset
$\Delta\subseteq\Omega$, the weight $z_i$ is used when $i\in\Delta$, and
$n_i$ is used when $i\notin\Delta$. The theorem asserts that $\Delta$ can be
chosen so that the resulting weighted abundance ratio is at least~$1$. In
the cubic case of Section~\ref{sec:imprimitive}, the theorem is applied to
the permutation group induced on the index set $I$ of the summands of an
imprimitive module.

\begin{theorem}\label{thm:weighted}
Let $P$ be a finite solvable transitive permutation group on $\Omega$.
For each $i\in\Omega$, let $z_i,n_i$ be positive real numbers satisfying
$z_i n_i^{3}\ge1$. Then there exists $\Delta\subseteq\Omega$ such that
$$
  \rho_P(\Delta)\prod_{i\in\Delta}z_i
  \prod_{i\notin\Delta}n_i\ge1.
$$
\end{theorem}

Before proving the theorem, we establish several lemmas.

\begin{lemma}\label{lem:symmetrise}
Let $P$ be a finite transitive permutation group on $\Omega$. Suppose that
\begin{equation}\label{eq:one-parameter}
  \max_{\Delta\subseteq\Omega}
  \rho_P(\Delta)r^{\,|\Omega|-4|\Delta|}\ge1
\end{equation}
for every $r>0$. For each $i\in\Omega$, let $z_i,n_i$ be positive real
numbers satisfying $z_i n_i^3\ge1$. Then there exists
$\Delta\subseteq\Omega$ such that
$$
  \rho_P(\Delta)\prod_{i\in\Delta}z_i
  \prod_{i\notin\Delta}n_i\ge1.
$$
\end{lemma}

\begin{proof}
Write $d=|\Omega|$. Since $z_i\ge n_i^{-3}$ for every $i\in\Omega$, it
is enough to prove the conclusion when $z_i=n_i^{-3}$ for every $i\in\Omega$.
Define
$$
  r=\left(\prod_{i\in\Omega}n_i\right)^{1/d}.
$$
By~\eqref{eq:one-parameter}, choose $\Delta\subseteq\Omega$ such that
$\rho_P(\Delta)r^{\,d-4|\Delta|}\ge1$.
The setwise stabilisers of $\Delta$ and $\Delta^g$ are conjugate for every
$g\in P$, so $\rho_P(\Delta^g)=\rho_P(\Delta)$. Since $P$ is transitive, the
number of elements $g\in P$ such that $i\in\Delta^g$ is the same for every
$i\in\Omega$. Counting the pairs $(i,g)\in\Omega\times P$ with
$i\in\Delta^g$ shows that this number is $|P|\,|\Delta|/d$.
Consequently
$$
\prod_{g\in P}\left(\rho_P(\Delta^g)\prod_{i\in\Delta^g}n_i^{-3}
\prod_{i\notin\Delta^g}n_i\right)
=\rho_P(\Delta)^{|P|}r^{\,|P|(d-4|\Delta|)}
=\left(\rho_P(\Delta)r^{\,d-4|\Delta|}\right)^{|P|}\ge1.
$$
All the factors are positive. Since their product is at least $1$, choose
$g\in P$ for which the factor indexed by $g$ is at least $1$. The inequalities
$z_i\ge n_i^{-3}$ then show that $\Delta^g$ satisfies the conclusion.
\end{proof}

We next formulate the local condition used to prove~\eqref{eq:one-parameter}
for $r\ge1$.

\begin{definition}\label{def:cubic-pair}
Let $J$ be a primitive permutation group on a finite set $D$. Let
$\mathcal Z$ and $\mathcal N$ be nonempty disjoint collections of $J$-orbits
on the power set $\pow(D)$. For each chain
$L\trianglelefteq K\trianglelefteq J$ and each $r>0$, define
$$
  Z_L(r)=\max_{\substack{\Delta\subseteq D\\ \Delta^J\in\mathcal Z}}
  \rho_L(\Delta)\,r^{\,|D|-4|\Delta|},
  \qquad
  N_L(r)=\max_{\substack{\Delta\subseteq D\\ \Delta^J\in\mathcal N}}
  \rho_L(\Delta)\,r^{\,|D|-4|\Delta|}.
$$
We call the ordered pair $(\mathcal Z,\mathcal N)$ a {\it cubic pair} for $J$
at $r$ if
$Z_L(r)N_L(r)^{3}\ge1$ for every such chain.
\end{definition}

\begin{lemma}\label{lem:cubic-pair-transfer}
Let $P$ be a finite transitive permutation group on $\Omega$. Let
$D_1,\ldots,D_m$ be an unrefinable block system. Set
$I=\{1,\ldots,m\}$, where $i$ indexes $D_i$, and let
$\overline P\leq\Sym(I)$ be the group induced by $P$ on $I$. Let $J$ be the
local group on $D_1$, let $r>0$, and suppose that
$(\mathcal Z,\mathcal N)$ is a cubic pair for $J$ at $r$.
Suppose that, for every choice of positive real numbers $z_i,n_i$, $i\in I$,
satisfying $z_i n_i^3\geq1$ for every $i\in I$, there exists
$\Gamma\subseteq I$ such that
\begin{equation}\label{eq:cubic-transfer-top}
  \rho_{\overline P}(\Gamma)
  \prod_{i\in\Gamma}z_i
  \prod_{i\notin\Gamma}n_i\geq1.
\end{equation}
Then there exists $\Delta\subseteq\Omega$ such that $\rho_P(\Delta)r^{\,|\Omega|-4|\Delta|}\ge1$.
\end{lemma}

\begin{proof}
For each $i\in I$, let $L_i\trianglelefteq K_i\trianglelefteq J_i$ be the
coordinate chain at $i$. The local identifications in
Section~\ref{sec:coordinate} allow us to use $\mathcal Z$ and $\mathcal N$ on
each $D_i$. Set $z_i=Z_{L_i}(r)$ and $n_i=N_{L_i}(r)$.
The cubic pair condition gives $z_i n_i^3\geq1$, so the hypothesis on
$\overline P$ supplies $\Gamma\subseteq I$ satisfying
\eqref{eq:cubic-transfer-top}. For each $i\in I$, choose a subset of $D_i$
attaining $Z_{L_i}(r)$ if
$i\in\Gamma$, and one attaining $N_{L_i}(r)$ otherwise. Let $\Delta$ be their
union, and let
$T\leq\overline P$ be the subgroup preserving the colours of the corresponding
coordinate tuple. If $\Gamma=\varnothing$ or $\Gamma=I$, then
$T\leq\overline{P}_\Gamma$ trivially. Otherwise the two orbit families are
disjoint, and again $T\leq\overline{P}_\Gamma$. In either case,
$|\overline{P}:T|/\nu(T)\ge\rho_{\overline{P}}(\Gamma)$. Multiplying~\eqref{eq:coord-clean} by
$r^{\,|\Omega|-4|\Delta|}$ and grouping the weight block by block gives
$$\rho_P(\Delta)r^{\,|\Omega|-4|\Delta|}\ge\rho_{\overline{P}}(\Gamma)\prod_{i\in\Gamma}z_i\prod_{i\notin\Gamma}n_i\ge1,$$
which proves the assertion.
\end{proof}

We next construct cubic pairs for the primitive groups other than $C_2$. We
begin with the case in which a regular subset is available.

\begin{lemma}\label{lem:three-label}
Let $J$ be a nontrivial primitive permutation group on a finite set $D$. Let
$\Delta_0\subseteq D$ be a regular subset with
$|\Delta_0|=a\ne|D|/2$. Then $\mathcal Z=\{(\Delta_0^c)^J\}$ and
$\mathcal N=\{\Delta_0^J,\,\{\varnothing\}\}$ form a cubic pair for $J$ at
every $r>0$.
\end{lemma}

\begin{proof}
The three orbits are distinct. Fix a chain
$L\trianglelefteq K\trianglelefteq J$. Regularity gives
$\rho_L(\Delta_0)=\rho_L(\Delta_0^c)=|L|$, while
$\rho_L(\varnothing)=\nu(L)^{-1}$. With
$s=r^{4a}$,
$$
  Z_L(r)=|L|\,r^{4a-3|D|},\qquad N_L(r)=\max\{|L|\,r^{|D|-4a},\,\nu(L)^{-1}r^{|D|}\},
$$
so
$$
  Z_L(r)N_L(r)^{3}=\max\Big\{\frac{|L|^4}{s^{2}},\ \frac{|L|}{\nu(L)^{3}}\,s\Big\}.
$$
If $s\le |L|\nu(L)$, the first term is at least $(|L|/\nu(L))^2$. If
$s\ge |L|\nu(L)$, the second term is at least $(|L|/\nu(L))^2$. As
$\nu(L)\le |L|$, the product is at least $1$.
\end{proof}

\begin{lemma}\label{lem:weighted-small}
Let $J$ be a solvable primitive permutation group on a finite set $D$ with
$|D|\in\{3,4,5,7,8,9\}$. Then $J$ admits a cubic pair for every $r\ge1$.
The following choices suffice. When $|D|=3$ or $4$, let $\mathcal O_k$
denote the $J$-orbit consisting of the subsets of $D$ of size $k$.
\begin{enumerate}
\item If $|D|=3$, then $J\cong C_3$ or $S_3$. Take
$\mathcal Z=\{\mathcal O_0\}$ and $\mathcal N=\{\mathcal O_1\}$.
\item If $|D|=4$, then $J\cong A_4$ or $S_4$. Take
$\mathcal Z=\{\mathcal O_0,\mathcal O_2\}$ and
$\mathcal N=\{\mathcal O_1\}$.
\item If $|D|\in\{5,7,8,9\}$, let $\mathcal A$ and $\mathcal B$ be
distinct chain-abundant $J$-orbits on $\pow(D)$ such that the subsets in
$\mathcal A$ and $\mathcal B$ have sizes $a_Z$ and $a_N$, respectively, and
$a_Z+3a_N\leq|D|$. Such orbits exist by
Lemma~\textup{\ref{lem:deg5789}}. Set $\mathcal Z=\{\mathcal A\}$ and
$\mathcal N=\{\mathcal B\}$.
\end{enumerate}
\end{lemma}

\begin{proof}
In~(1), the powers of $r$ occurring in $Z_L(r)$ and $N_L(r)$ are $r^3$
and $r^{-1}$, respectively. They therefore cancel in
$Z_L(r)N_L(r)^3$. The value is $1,9,9/8$ for $L=1,C_3,S_3$
respectively. In~(2), set $\alpha=\max_{|\Delta|=2}\rho_L(\Delta)$ and
$\beta=\max_{|\Delta|=1}\rho_L(\Delta)$. With $s=r^4$, one has
$N_L(r)=\beta$ and $Z_L(r)=\max\{\alpha/s,s/\nu(L)\}\ge\sqrt{\alpha/\nu(L)}$. Hence $Z_L(r)N_L(r)^{3}\ge\beta^{3}\sqrt{\alpha/\nu(L)}$, and it suffices that $\alpha\beta^{6}\ge
\nu(L)$. The five possible bottom groups give
$$
\renewcommand{\arraystretch}{1.15}
\begin{array}{c|ccccc}
 L & 1 & C_2 & V_4 & A_4 & S_4\\\hline
 \alpha\beta^{6} & 1 & 128 & 4\,096 & 4\,096/243 & 2\,048/243\\
 \nu(L) & 1 & 2 & 4 & 4 & 8
\end{array}
$$
and $\alpha\beta^{6}\ge \nu(L)$ in every column. In~(3), both orbits are
chain-abundant. Hence each $\rho_L$ may be taken at least $1$, and then
$Z_L(r)N_L(r)^{3}\ge r^{(|D|-4a_Z)+3(|D|-4a_N)}=r^{4(|D|-a_Z-3a_N)}\ge1$ for $r\ge1$,
the sizes being $(a_Z,a_N)=(2,1)$ or $(3,2)$ as in
Lemma~\ref{lem:deg5789}.
\end{proof}

\begin{lemma}\label{lem:c2-weighted}
Let $P$ be a finite transitive permutation group on $\Omega$. Let
$D_1,\ldots,D_m$ be a block system with $|D_i|=2$ for every $i$, and let
$I=\{1,\ldots,m\}$, where $i$ indexes $D_i$. Suppose that the local group on $D_1$ is isomorphic to
$C_2$, and let
$\overline P\leq\Sym(I)$ be the group induced by $P$ on $I$. If $\overline P$ has
an abundant subset, then, for every $r>0$, there exists
$\Delta\subseteq\Omega$ such that
$$
  \rho_P(\Delta)r^{\,|\Omega|-4|\Delta|}\ge1.
$$
\end{lemma}

\begin{proof}
Choose $\Gamma\subseteq I$ with
$\rho_{\overline P}(\Gamma)\geq1$. For each $i\in I$, let $L_i$ be the
bottom group at $i$ and choose a singleton $E_i\subseteq D_i$. Set
$$
  \Delta=\bigcup_{i\in\Gamma}E_i,
  \qquad
  \Delta'=\bigcup_{i\notin\Gamma}E_i.
$$
Let $T,T'\leq\overline P$ be the subgroups preserving the colours of the
coordinate tuples corresponding to $\Delta$ and $\Delta'$, respectively.
In the first tuple, the singleton entries are indexed by
$\Gamma$, while in the second they are indexed by $I\setminus\Gamma$.
Thus $T,T'\leq\overline P_\Gamma$. By~\eqref{eq:top-monotone},
$$
  \frac{|\overline P:T|}{\nu(T)}
  \geq\rho_{\overline P}(\Gamma),
  \qquad
  \frac{|\overline P:T'|}{\nu(T')}
  \geq\rho_{\overline P}(\Gamma),
$$
so both top factors are at least $1$.

Lemma~\ref{lem:c2-balance} gives
$$
 \left(\prod_{i\in I}\rho_{L_i}(\Delta\cap D_i)\right)
 \left(\prod_{i\in I}\rho_{L_i}(\Delta'\cap D_i)\right)=1.
$$
Moreover, $|\Omega|=2m$, $|\Delta|=|\Gamma|$ and
$|\Delta'|=m-|\Gamma|$. Hence the corresponding weight factors are
$$
 r^{2m-4|\Gamma|}
 \qquad\text{and}\qquad
 r^{2m-4(m-|\Gamma|)}=r^{-(2m-4|\Gamma|)},
$$
which are reciprocal. Combining each local product with its weight
factor therefore gives two positive numbers whose product is $1$.

Apply Theorem~\ref{thm:coordinate} to the two coordinate tuples and
multiply the resulting inequalities by their respective weight factors.
Consequently, the asserted inequality holds for at least one of $\Delta$
and $\Delta'$.
\end{proof}

\begin{proof}[Proof of Theorem~\ref{thm:weighted}]
We prove the theorem by induction on $d=|\Omega|$. If $d=1$, write
$\Omega=\{i\}$. The group $P$ is trivial. If $n_i\ge1$, then
$\Delta=\varnothing$ gives the required inequality. If $n_i<1$, then
$z_i\ge n_i^{-3}>1$, and $\Delta=\Omega$ gives the required inequality.

Assume that $d\ge2$ and that the theorem holds for smaller degrees. By
Lemma~\ref{lem:symmetrise}, it suffices to prove~\eqref{eq:one-parameter} for
every $r>0$. Suppose first that $0<r\leq1$. By
Theorem~\ref{thm:top-hni}, choose $\Delta\subseteq\Omega$ with
$\rho_P(\Delta)\geq1$. Since
$P_\Delta=P_{\Delta^c}$, we have
$\rho_P(\Delta)=\rho_P(\Delta^c)$. Replacing $\Delta$ by its complement if
necessary, we may assume that $|\Delta|\ge d/2$. It follows that
$d-4|\Delta|\le-d\le0$, and hence
$\rho_P(\Delta)r^{\,d-4|\Delta|}\ge1$.

Now fix $r\ge1$. Every solvable primitive permutation group of degree at
least three has a cubic pair at $r$. For degree greater than nine, use the
regular subset supplied by Lemma~\ref{lem:large} and apply
Lemma~\ref{lem:three-label}. In degrees $3,4,5,7,8,9$, such a pair is
supplied by Lemma~\ref{lem:weighted-small}.

Suppose first that $P$ is primitive. If $P\cong C_2$, take
$D_1=\Omega$. The local group is $P\cong C_2$, and the group
$\overline P$ induced on $I=\{1\}$ is trivial, so it has an abundant subset.
Hence Lemma~\ref{lem:c2-weighted} applies. Otherwise choose a cubic pair
for $P$ at $r$ as above. Taking $L=K=P$ in its defining inequality gives
$Z_P(r)N_P(r)^3\ge1$. Hence
$\max\{Z_P(r),N_P(r)\}\ge1$. Choose a subset $\Delta$ attaining whichever
of $Z_P(r)$ and $N_P(r)$ is larger. Then
$\rho_P(\Delta)r^{\,d-4|\Delta|}\ge1$, so~\eqref{eq:one-parameter} holds.

Now suppose that $P$ is imprimitive. Choose an unrefinable block system and
number its blocks $D_i$, where $I=\{1,\ldots,m\}$ is the index set. Let $J$
be the local group on $D_1$, and let $\overline P\leq\Sym(I)$ be the group
induced by $P$ on $I$. If $J\cong C_2$,
Theorem~\ref{thm:top-hni} provides an abundant
subset of $I$ for $\overline P$, so Lemma~\ref{lem:c2-weighted} applies. Otherwise
$J$ has a cubic pair at $r$ by the preceding paragraph. The group
$\overline P$ is solvable and transitive of degree $m<d$, so the induction
hypothesis gives the weighted hypothesis of
Lemma~\ref{lem:cubic-pair-transfer}. That lemma gives
$\Delta\subseteq\Omega$ with
$\rho_P(\Delta)r^{\,d-4|\Delta|}\ge1$.

Together with the case $0<r\le1$, this proves~\eqref{eq:one-parameter} for
every $r>0$. Lemma~\ref{lem:symmetrise} completes the proof.
\end{proof}

\section{Primitive linear groups}
\label{sec:primitive}

We now prove the abundance theorem for faithful irreducible primitive
modules. The nonmetacyclic case follows from the classification of actions
without a regular orbit~\cite{HY23} and the finite verification in
Appendix~\ref{sec:computations}. The remaining metacyclic case is handled by
a semilinear argument.

\subsection{The nonmetacyclic case}

We require only the primitive part of the classification of residual actions.

\begin{theorem}[Primitive residual actions]
\label{thm:residual-classification}
Let $H$ be a finite solvable group acting faithfully, irreducibly and
primitively on a finite $\mathbb{F}_p$-space $V$. If $H$ is not
metacyclic and has no regular orbit, then its action is linearly equivalent
to one represented in the linearly primitive collection described in
Appendix~\ref{sec:computations}.
\end{theorem}

\begin{proof}
Every primitive irreducible linear action is quasi-primitive
\cite[p.~27]{ManzWolf93}, so the classification of residual actions applies.
Holt and Yang~\cite{HY23} gave the original classification and accompanying
matrix data. The corrected Holt--Yang data are proved complete
in~\cite{ZhangRegularOrbits26} and described in
Appendix~\ref{sec:computations}.
\end{proof}

\begin{corollary}\label{cor:primitive-metacyclic}
Let $H$ be a finite solvable group acting faithfully, irreducibly and
primitively on a finite $\mathbb{F}_p$-space $V$. If $V$ contains no
abundant vector, then $H$ must be metacyclic.
\end{corollary}

\begin{proof}
If $H$ had a regular orbit on $V$, then some $v\in V$ would satisfy
$H_v=1$, and hence $\rho_H(v)=|H|$. Thus $v$ would be abundant. Therefore
$H$ has no regular orbit. If $H$ were not metacyclic,
Theorem~\ref{thm:residual}\textup{(1)} would again give an abundant vector.
Hence $H$ is metacyclic.
\end{proof}

\subsection{The metacyclic case}
\label{sec:semilinear}

We now consider faithful irreducible primitive metacyclic groups. We
first prove a theorem on regular pairs for primitive subgroups of
$\GammaL(1,Q)$, and then apply it to the metacyclic groups. Related
semilinear estimates are used again in Lemma~\ref{lem:branch-semilinear}.

Let $Q=p^{f}$ and identify the natural $f$-dimensional $\mathbb{F}_p$-space with the
field $\mathbb{F}_Q$. Write $A=\mathbb{F}_Q^{\times}$ for the group of scalar
transformations, and
$$
  \GammaL(1,Q)
  =\{z\mapsto az^\sigma\mid a\in A,\ 
      \sigma\in\Gal(\mathbb{F}_Q/\mathbb{F}_p)\}
  =A\rtimes\Gal(\mathbb{F}_Q/\mathbb{F}_p).
$$
This is the full semilinear group
of $\mathbb{F}_Q$. The notation $\Gamma(p^f)$ used in~\cite{HY23,Yang10,Yang11b,YangVV22} denotes the same group. For $K\leq\GammaL(1,Q)$, let $n_{\rm reg}(K)$ denote the
number of regular orbits of the diagonal action of $K$ on
$\mathbb{F}_Q\times\mathbb{F}_Q$.

The natural homomorphism with kernel $A$ is given by
$$
  \GammaL(1,Q)\longrightarrow\Gal(\mathbb{F}_Q/\mathbb{F}_p),
  \qquad (z\mapsto az^\sigma)\longmapsto\sigma.
$$

For a prime power $q$ and a positive integer $t$, let $R(q,t)$ denote the
number of regular orbits of $\Gal(\mathbb{F}_{q^t}/\mathbb{F}_q)$ on
$\mathbb{F}_{q^t}$.

\begin{lemma}\label{lem:semilinear-count}
Let $K\leq\GammaL(1,Q)$, where $Q=p^f$, and set $U=K\cap A$.
Let $C\cong K/U$ be the image of $K$ in
$\Gal(\mathbb{F}_Q/\mathbb{F}_p)$. Write $t=|C|$ and $q=p^{f/t}$. Then
$$
  n_{\rm reg}(K)\ \ge\ |A:U|\,R(q,t).
$$
\end{lemma}

\begin{proof}
The group $\Gal(\mathbb{F}_Q/\mathbb{F}_p)$ is cyclic of order $f$.
Consequently $t\mid f$, $Q=q^t$ and
$C=\Gal(\mathbb{F}_Q/\mathbb{F}_q)$. Thus $C$ has $R(q,t)$ regular orbits
on $\mathbb{F}_Q$. Let $\mathcal R$ be one of these orbits, and define
$$
  S_{\mathcal R}=\{(x,\xi x):x\in A,\ \xi\in\mathcal R\}.
$$
For $x\in A$ and $\xi\in\mathcal R$, suppose that an element
$z\mapsto az^\sigma$ of $K$ fixes both $x$ and $\xi x$. Then
$ax^\sigma=x$ and $a(\xi x)^\sigma=\xi x$. Since $\sigma$ is a field
automorphism, $(\xi x)^\sigma=\xi^\sigma x^\sigma$. As $a$, $x$ and
$x^\sigma$ are nonzero, the two equalities give
$$
  \frac{a\xi^\sigma x^\sigma}{ax^\sigma}=\frac{\xi x}{x},
$$
and hence $\xi^\sigma=\xi$. Since $\mathcal R$ is regular,
$\sigma=1$, and then $ax=x$ forces $a=1$. Hence $(x,\xi x)$ is a regular pair.

The set $S_{\mathcal R}$ is $K$-invariant: the map $z\mapsto az^\sigma$
sends $(x,\xi x)$ to
$(ax^\sigma,\xi^\sigma ax^\sigma)$, where $ax^\sigma\in A$ and
$\xi^\sigma\in\mathcal R$. Since
$\mathcal R$ is regular and $|C|=t$, the orbit--stabiliser theorem gives
$|\mathcal R|=t$. Hence $|S_{\mathcal R}|=|A|t$, and all its elements are
regular. Each of its $K$-orbits
therefore has $|K|=|U|t$ elements, so it contains $|A:U|$ orbits of regular
pairs. The sets $S_{\mathcal R}$ arising from distinct choices of
$\mathcal R$ are disjoint,
and summing over the $R(q,t)$ choices proves the result.
\end{proof}

\begin{lemma}\label{lem:R-small}
Let $q$ be a prime power and let $t\ge1$ be an integer. Then $R(q,t)<5$ only
for
$$
  (q,t)\in\{(2,1),(2,2),(2,3),(2,4),(3,1),(3,2),(4,1)\}.
$$
\end{lemma}

\begin{proof}
An element $\xi\in\mathbb{F}_{q^t}$ has trivial stabiliser in
$\Gal(\mathbb{F}_{q^t}/\mathbb{F}_q)$ precisely when it lies in no proper
intermediate subfield, or equivalently when
$\mathbb{F}_q(\xi)=\mathbb{F}_{q^t}$.
For each divisor $d$ of $t$, let $N_d$ be the number of elements
$\xi\in\mathbb{F}_{q^d}$ satisfying
$\mathbb{F}_q(\xi)=\mathbb{F}_{q^d}$. Every element of
$\mathbb{F}_{q^d}$ generates $\mathbb{F}_{q^e}$ over $\mathbb{F}_q$ for a
unique divisor $e$ of $d$, and hence
$$
  q^d=\sum_{e\mid d}N_e.
$$
M\"obius inversion gives
$$
  N_t=\sum_{d\mid t}\mu(t/d)q^d,
$$
where $\mu$ is the M\"obius function. Thus $N_t$ is the number of elements
with trivial stabiliser. Every regular orbit has length $t$, and therefore
$$
  R(q,t)=\frac{N_t}{t}
        =\frac1t\sum_{d\mid t}\mu(t/d)q^d.
$$

For $t=1$, we have $R(q,1)=q$. As $q$ is a prime power,
$R(q,1)<5$ precisely when $q\in\{2,3,4\}$. For $t=2,3,4$,
respectively,
$$
R(q,2)=\frac{q^2-q}{2},\qquad
R(q,3)=\frac{q^3-q}{3},\qquad
R(q,4)=\frac{q^4-q^2}{4}.
$$
Thus $R(q,2)<5$ precisely when $q\in\{2,3\}$, while
$R(q,3)<5$ and $R(q,4)<5$ precisely when $q=2$. For $t\ge5$, every element of
$\mathbb{F}_{q^t}$ that does not generate the extension over $\mathbb{F}_q$
lies in a subfield $\mathbb{F}_{q^{e}}$ with $e\mid t$ and $e\le t/2$. Hence
 $$
 tR(q,t)\ge q^{t}-\sum_{1\le e\le\lfloor t/2\rfloor}q^{e}.
 $$
 For $q\ge3$ the subtracted sum is below $q^{t}/2$, whence $tR(q,t)>q^{t}/2\ge5t$ and
$R(q,t)\ge5$. For $q=2$ the same bound gives
$$
 tR(2,t)\ge2^{t}-2^{\lfloor t/2\rfloor+1}+2.
$$
The difference
$$
 2^{t}-2^{\lfloor t/2\rfloor+1}+2-5t
$$
is positive for $t=5,6$ and increases thereafter separately along the odd and
even values of $t$. Hence $R(2,t)\ge5$ for all $t\ge5$.
\end{proof}

\begin{theorem}\label{thm:semilinear}
Let $K\le\GammaL(1,Q)$ act primitively on the
$\mathbb{F}_p$-space $\mathbb{F}_Q$. Then
$K$ has a
regular pair, and has at least five orbits of regular pairs except for the six groups
$$
  (Q,K)=(2,1),\ (3,C_2),\ (4,\GammaL(1,4)),\ (8,\GammaL(1,8)),\
  (9,\GammaL(1,9)),\ (16,\GammaL(1,16)),
$$
whose values of $n_{\rm reg}(K)$ are $4,4,1,2,3,3$ respectively.
\end{theorem}

\begin{proof}
Suppose $n_{\rm reg}(K)<5$. Set $U=K\cap A$, and let $C$ be the
image of $K$ under the natural homomorphism to
$\Gal(\mathbb{F}_Q/\mathbb{F}_p)$. Let $\mathbb{F}_q$ be the subfield of
$\mathbb{F}_Q$ fixed pointwise by $C$, and set
$t=[\mathbb{F}_Q:\mathbb{F}_q]$. By Lemma~\ref{lem:semilinear-count},
$|A:U|R(q,t)<5$. Thus $R(q,t)<5$, so $(q,t)$ is one of the seven
pairs in Lemma~\ref{lem:R-small}.

If $t=1$, then $K=U\le A$. For $q=2$, the bound forces the trivial group on
$\mathbb{F}_2$. For $q=3$, it forces $U=A\cong C_2$. For $q=4$, it forces
$U=A\cong C_3$, but every ordered pair other than $(0,0)$ is then regular.
There are therefore $(16-1)/3=5$ orbits of regular pairs, contrary to
$n_{\rm reg}(K)<5$. If $(q,t)=(2,2)$, then $Q=4$ and
$A\cong C_3$. Suppose that $U=1$. Then $K$ has order $2$, and its
nonidentity element has the form $z\mapsto az^2$ for some
$a\in\mathbb{F}_4^\times$. This element fixes $a^{-1}$, since
$a(a^{-1})^2=a^{-1}$.
Thus $\{0,a^{-1}\}$ is a proper nonzero $K$-invariant
$\mathbb{F}_2$-subspace of $\mathbb{F}_4$, contrary to the primitivity of $K$. Therefore
$U=A$ and $K=\GammaL(1,4)$. The pairs
$(2,3),(2,4),(3,2)$ correspond to $Q=8,16,9$. In each,
$|A:U|\,R(q,t)\le n_{\rm reg}(K)<5$
forces $|A:U|<5/R(q,t)$, and as $|A:U|$ divides $|A|$ this leaves only $|A:U|=1$:
for $Q=8$, $|A|=7$ and $R=2$, whence $|A:U|<5/2$. For $Q=16$, $|A|=15$
and $R=3$, whence $|A:U|<5/3$. For $Q=9$, $|A|=8$ and $R=3$, whence
$|A:U|<5/3$. In every case, the next divisor of $|A|$ exceeds the bound.
Hence $U=A$ and $K$ is
$\GammaL(1,8),\GammaL(1,16),\GammaL(1,9)$ respectively. This leaves exactly
the six named groups. Their values of $n_{\rm reg}$ and the existence of a
regular pair in each are verified in Appendix~\ref{sec:comp-semilinear}.
\end{proof}

\begin{corollary}\label{cor:metacyclic}
Let $W$ be a finite $\mathbb{F}_p$-space and let $K\leq\GL(W)$ be
metacyclic, irreducible and primitive. Then $K$ has a regular pair on $W$.
\end{corollary}

\begin{proof}
A metacyclic group is abelian-by-nilpotent. By a theorem of Kov\'acs
\cite[Theorem~5.3.1]{Short92}, $K$ therefore normalises a Singer cycle.
The normaliser of a Singer cycle is the semilinear group
\cite[Theorem~2.3.5]{Short92}. Hence Theorem~\ref{thm:semilinear} applies.
\end{proof}

\begin{theorem}\label{thm:primitive-action}
Let $H$ be a finite solvable group, and let $V$ be a faithful irreducible
primitive $H$-module. Then $V$ contains an abundant vector.
\end{theorem}

\begin{proof}
Suppose that $V$ contains no abundant vector. By
Corollary~\ref{cor:primitive-metacyclic}, $H$ is metacyclic.
Corollary~\ref{cor:metacyclic} gives a regular pair, and
Lemma~\ref{lem:regular-pair} then gives an abundant vector, a contradiction.
\end{proof}

\section{Local alternatives for imprimitive modules}
\label{sec:local-alternatives}

Theorem~\ref{thm:primitive-action} gives an abundant vector in every
faithful irreducible primitive module. For an imprimitive module, however,
an abundant vector for the group induced on one
summand is not enough: Theorem~\ref{thm:coordinate} involves every bottom
group arising from a coordinate chain. We therefore need a pair of local
orbits that works uniformly along every coordinate chain, without assuming
that the top group has a regular orbit on the power set of the index set. Such
a regular orbit need not exist for an arbitrary solvable transitive
group~\cite{Glu25}.

\begin{definition}\label{def:balanced-pair}\label{def:chain-label-linear}
Let $W$ be a finite vector space and let $J\leq\GL(W)$ be solvable and
primitive. We call an ordered pair $(\orbit_0,\orbit_1)$ of distinct
$J$-orbits on $W$ a {\it balanced pair} for $J$ if, for every chain
$L\trianglelefteq K\trianglelefteq J$, there exist $a\in\orbit_0$ and
$b\in\orbit_1$ such that $\rho_L(a)\rho_L(b)\ge1$. We call a $J$-orbit
$\orbit$ on $W$ a {\it chain-abundant orbit} if, for every chain
$L\trianglelefteq K\trianglelefteq J$, there exists $a\in\orbit$ such that
$\rho_L(a)\ge1$. In both conditions, the orbits are fixed, but the
representatives may vary with the chain.
\end{definition}

\begin{lemma}\label{lem:balanced-sources}
Let $W$ be a finite vector space and let $J\leq\GL(W)$ be solvable and
primitive. The following constructions give balanced pairs for $J$:
\begin{enumerate}
\item Any two distinct chain-abundant orbits $\orbit_0\ne\orbit_1$ form a
balanced pair.
\item If $v\in W$ is a nonzero regular vector, then $(v^J,\{0\})$ is a
balanced pair.
\item If $(a,b)\in W\times W$ is a regular pair and $a,b$ lie in distinct
$J$-orbits, then $(a^J,b^J)$ is a balanced pair.
\end{enumerate}
\end{lemma}

\begin{proof}
In~(1), for every chain there exist $a\in\orbit_0$ and $b\in\orbit_1$ such
that $\rho_L(a)\ge1$ and $\rho_L(b)\ge1$. Hence the product is at least $1$. In~(2), for the bottom group $L$ of any chain the
stabiliser $L_v=L\cap J_v=1$. Hence $\rho_L(v)=|L|$ and
$\rho_L(v)\,\rho_L(0)=|L|/\nu(L)\ge1$. In~(3), for the bottom group $L$ of any chain we have
$L_a\cap L_b=L\cap J_a\cap J_b=1$. Hence $(a,b)$ is a regular pair for $L$ and
$\rho_L(a)\,\rho_L(b)\ge1$ by Lemma~\ref{lem:regular-pair}. In each case
the two orbits are distinct.
\end{proof}

Let $W$ be a finite vector space and let $J\leq\GL(W)$ be solvable and
primitive. For each chain $L\trianglelefteq K\trianglelefteq J$, define
$z_L=\rho_L(0)=1/\nu(L)$ and $n_L=\max_{0\ne w\in W}\rho_L(w)$. Thus $z_L$
is the abundance ratio of the zero vector, while $n_L$ is the largest
abundance ratio among the nonzero vectors of $W$.

\begin{theorem}\label{thm:local-alternatives}
Let $J$ be a nontrivial faithful irreducible primitive solvable linear group on
$W$. Then either $J$ has a balanced pair, or $J$ is metacyclic with no regular
vector, or the action of $J$ on $W$ is linearly equivalent to one of the nine residual actions of
Theorem~\ref{thm:residual}\textup{(2)}. In the last case,
$z_Ln_L^3\geq1$ for every chain $L\trianglelefteq K\trianglelefteq J$.
\end{theorem}

\begin{proof}
If $J$ has a regular vector $v$, then Lemma~\ref{lem:balanced-sources}(2)
makes $(v^J,\{0\})$ a balanced pair. We may therefore suppose that $J$ has no
regular vector and is not metacyclic. By
Theorem~\ref{thm:residual}\textup{(2)}, either $J$ has a balanced pair,
or its action on $W$ is linearly equivalent to one of the nine residual actions and
$z_Ln_L^3\geq1$ for every chain $L\trianglelefteq K\trianglelefteq J$.
\end{proof}

\section{Imprimitive modules}
\label{sec:imprimitive}

\begin{definition}\label{def:unrefinable}
Let $V$ be a finite vector space and let $H\leq\GL(V)$. A {\it system of
imprimitivity} for $H$ on $V$ is a decomposition
$V=W_1\oplus\cdots\oplus W_m$ into nonzero subspaces such that $H$ permutes
$\{W_1,\ldots,W_m\}$
transitively. The system is {\it nontrivial} if $m>1$. It is {\it
unrefinable} if there is no system of imprimitivity
$V=U_1\oplus\cdots\oplus U_n$ with $n>m$ such that every $U_j$ is
contained in some $W_i$.
\end{definition}

\begin{lemma}\label{lem:unrefinable-blocks}
Let $H$ be a finite group acting faithfully, irreducibly and imprimitively
on a finite $\mathbb{F}_p$-space $V$, and let
$V=W_1\oplus\cdots\oplus W_m$ be a nontrivial unrefinable system of
imprimitivity. Let $J$ be the group induced on $W_1$ by its setwise
stabiliser $H_{W_1}$. Then $J$ is nontrivial, faithful, irreducible and
primitive.
\end{lemma}

\begin{proof}
The action of $J$ on $W_1$ is faithful by definition. For each $i$, choose
$t_i\in H$ with $W_1^{t_i}=W_i$, taking $t_1=1$.

Suppose first that $J=1$. Then $H_{W_1}$ fixes $W_1$ pointwise. Choose
$0\ne w_1\in W_1$ and set $w_i=w_1^{t_i}$. If $t_i'\in H$ also carries
$W_1$ to $W_i$, then $t_i(t_i')^{-1}\in H_{W_1}$, and hence
$w_1^{t_i}=w_1^{t_i'}$. Thus $w_i$ is independent of the choice of $t_i$.
If $W_i^h=W_j$, then both $t_ih$ and $t_j$ carry $W_1$ to $W_j$, so
$w_i^h=w_j$. Therefore $H$ fixes
$w=w_1+\cdots+w_m$. Since $W_1\oplus\cdots\oplus W_m$ is a direct sum,
$w\ne0$. As $m>1$, the line $\mathbb{F}_p w$ is a proper $H$-submodule of
$V$. This contradicts the irreducibility of $V$, so $J\ne1$.

Suppose that $0<U<W_1$ is $J$-invariant, and set $U_i=U^{t_i}$. The same
argument shows that $U_i$ is independent of the choice of $t_i$ and that
$H$ permutes $U_1,\ldots,U_m$. Hence
$U_1\oplus\cdots\oplus U_m$ is a proper nonzero $H$-submodule of $V$,
again contradicting irreducibility. Therefore $J$ is irreducible.

Finally, suppose that $J$ is imprimitive, and let
$W_1=U_1\oplus\cdots\oplus U_s$ be a nontrivial system of imprimitivity
for $J$. For each $i$, transport
this decomposition to $W_i$ using $t_i$. The resulting decomposition of
$W_i$ is independent of the choice of $t_i$, since $H_{W_1}$ permutes
$U_1,\ldots,U_s$. We obtain
$$
  V=\bigoplus_{i=1}^m\bigoplus_{j=1}^s U_j^{t_i}.
$$
Since $H$ is transitive on $\{W_1,\ldots,W_m\}$ and $J$ is transitive on
$\{U_1,\ldots,U_s\}$, the group $H$ permutes the $ms$ summands in this
decomposition transitively.
This is a system of imprimitivity that strictly refines
$W_1\oplus\cdots\oplus W_m$, contrary to unrefinability. Thus $J$ is
primitive.
\end{proof}

For the next three lemmas, let $H$ be a finite solvable group acting
faithfully, irreducibly and imprimitively on a finite $\mathbb{F}_p$-space
$V$. Fix a nontrivial unrefinable system
$V=W_1\oplus\cdots\oplus W_m$, set $I=\{1,\ldots,m\}$, where $i$
indexes $W_i$, and let $P\leq\Sym(I)$ be the group induced by $H$ on
$I$. Let $J$ be the group induced on $W_1$ by its setwise stabiliser
$H_{W_1}$. Then $P$ is solvable and transitive. The group $J$ is
solvable, and Lemma~\ref{lem:unrefinable-blocks} shows that it is
nontrivial, faithful, irreducible and primitive.

For each $i\in I$, choose $t_i\in H$ with $W_1^{t_i}=W_i$, taking
$t_1=1$. The map $w\mapsto w^{t_i^{-1}}$ from $W_i$ to $W_1$
identifies the local group $J_i$ with $J$. Under this identification, we
write each coordinate chain as
$L_i\trianglelefteq K_i\trianglelefteq J$, and abundance ratios are
unchanged. In coordinate $i$, we use the same symbol $\orbit$ for the
$J$-orbit on $W_1$ and its transported orbit $\orbit^{t_i}$ on $W_i$.
For every chain $L\trianglelefteq K\trianglelefteq J$, set
$z_L=\rho_L(0)=1/\nu(L)$ and
$n_L=\max_{0\ne w\in W_1}\rho_L(w)$.

\begin{lemma}[Balanced case]\label{lem:branch-balanced}
If $J$ has a balanced pair, then $V$ contains an abundant vector.
\end{lemma}

\begin{proof}
Let $(\orbit_0,\orbit_1)$ be a balanced pair for $J$. By Theorem~\ref{thm:top-hni}
choose $\Gamma\subseteq I$ with $\rho_P(\Gamma)\ge1$. Since $P$ is
transitive on $I$ and $m>1$, it is nontrivial and
$\rho_P(\varnothing)=\rho_P(I)=1/\nu(P)<1$. Therefore $\Gamma$ is
neither empty nor all of $I$. For each coordinate $i\in I$, with bottom group
$L_i$ in its coordinate chain, pick representatives $a_i\in\orbit_0$ and
$b_i\in\orbit_1$ with $\rho_{L_i}(a_i)\,\rho_{L_i}(b_i)\ge1$. Consider two
choices of local vectors. In the first, use $a_i$ when $i\in\Gamma$ and
$b_i$ when $i\notin\Gamma$. In the second, use $b_i$ when $i\in\Gamma$
and $a_i$ when $i\notin\Gamma$. The corresponding products of
local ratios,
$$
  \prod_{i\in\Gamma}\rho_{L_i}(a_i)\prod_{i\notin\Gamma}\rho_{L_i}(b_i)
  \qquad\text{and}\qquad
  \prod_{i\in\Gamma}\rho_{L_i}(b_i)\prod_{i\notin\Gamma}\rho_{L_i}(a_i),
$$
multiply to $\prod_{i\in I}\rho_{L_i}(a_i)\,\rho_{L_i}(b_i)\ge1$. Hence at least one
of them is $\ge1$. Let $x$ be the corresponding global vector. By
Section~\ref{sec:coordinate}, the entries chosen from the distinct
$J$-orbits $\orbit_0$ and $\orbit_1$ have different colours. In the first
choice, entries from $\orbit_0$ occur precisely at the coordinates indexed
by $\Gamma$. In the second choice, they occur precisely at the coordinates
indexed by $\Gamma^c$. In either case write $T=T_x$. It preserves $\Gamma$
in the first choice and $\Gamma^c$ in the second. Since
$P_\Gamma=P_{\Gamma^{c}}$, we have $T\le P_\Gamma$ in
either case and, by~\eqref{eq:top-monotone},
$|P:T|/\nu(T)\ge\rho_P(\Gamma)\ge1$. The coordinate chain
inequality~\eqref{eq:coord-clean} then gives
$\rho_H(x)\ge(|P:T|/\nu(T))\prod_i\rho_{L_i}(x_i)\ge1$.
\end{proof}

\begin{lemma}[Cubic condition]\label{lem:branch-profile}
If $z_{L_i}n_{L_i}^{3}\geq1$ for every $i\in I$, then $V$ contains an
abundant vector.
\end{lemma}

\begin{proof}
Apply Theorem~\ref{thm:weighted} to $P$, taking $z_i=z_{L_i}$ and
$n_i=n_{L_i}$ for each coordinate $i$, to obtain $\Gamma\subseteq I$ with
$\rho_P(\Gamma)\prod_{i\in\Gamma}z_{i}\prod_{i\notin\Gamma}n_{i}\ge1$. Place
the zero vector on the summands indexed by $\Gamma$, and on every other
summand choose a nonzero vector attaining $n_{L_i}$. Let $x$ be their sum and
write $T=T_x$. If $\Gamma=\varnothing$ or $\Gamma=I$, then
$T\leq P_\Gamma=P$. Otherwise a zero entry and a nonzero entry have different
colours, so again $T\leq P_\Gamma$. Hence
Theorem~\ref{thm:coordinate} and \eqref{eq:top-monotone} give
\begin{equation*}
 \rho_H(x)\ge \rho_P(\Gamma)
 \prod_{i\in\Gamma}z_i\prod_{i\notin\Gamma}n_i\ge1.\qedhere
\end{equation*}
\end{proof}

\begin{lemma}[Semilinear case]\label{lem:branch-semilinear}
If $J$ is metacyclic and has no regular vector in $W_1$, then
$z_Ln_L^{3}\geq1$ for every chain
$L\trianglelefteq K\trianglelefteq J$.
\end{lemma}

\begin{proof}
Write $Q=|W_1|$. As in the proof of Corollary~\ref{cor:metacyclic}, we may
identify $W_1$ with $\mathbb{F}_Q$ in such a way that
$$
  J\leq\GammaL(1,Q)
   =\mathbb{F}_Q^\times\rtimes\Gal(\mathbb{F}_Q/\mathbb{F}_p).
$$
Here $Q\ge4$, since for $Q\in\{2,3\}$ every subgroup of $\GammaL(1,Q)$ has a
regular nonzero vector, whereas $J$ has none. Fix a chain
$L\trianglelefteq K\trianglelefteq J$, and let
$A_L=L\cap\mathbb{F}_Q^{\times}$ be its scalar part and let $C_L$ be the image
of $L$ under the natural homomorphism to
$\Gal(\mathbb{F}_Q/\mathbb{F}_p)$. Both $A_L$ and $C_L$ are cyclic, and $L$
fits into the exact sequence
$$
  1\longrightarrow A_L\longrightarrow L\longrightarrow C_L\longrightarrow 1,
$$
which need not split.

For $v\ne0$ no nonidentity scalar fixes $v$. Hence the stabiliser $L_v$ meets
$\mathbb{F}_Q^{\times}$ trivially and embeds in the cyclic group $C_L$. In particular, $L_v$
is cyclic, $\nu(L_v)=|L_v|$ and $\rho_L(v)=|L|/|L_v|^{2}$. Writing
$s_L=\min_{v\ne0}|L_v|$ we obtain $n_L=|L|/s_L^{2}$, and since $\nu(L)\le|L|$,
$$
  z_Ln_L^{3}=\frac{|L|^{3}}{s_L^{6}\,\nu(L)}\ \ge\ \frac{|L|^{2}}{s_L^{6}}.
$$
If $s_L=1$ this yields $z_Ln_L^{3}\ge|L|^{2}\ge1$.

Assume now $s_L\ge2$, so that every nonzero vector is fixed by a nonidentity element
of $L$. Such an element has the form $z\mapsto az^\sigma$ with $\sigma\ne1$,
because a nonidentity scalar fixes no nonzero vector. Fix $\sigma\ne1$ in
$C_L$, and let $c=|\sigma|$. The scalar coefficients $a$ for which the map
$z\mapsto az^\sigma$ belongs to $L$ form a single coset $a_\sigma A_L$ in
$\mathbb{F}_Q^{\times}$, and this map fixes $v\ne0$ precisely when
$v^{1-\sigma}=a$, where
$v^{1-\sigma}:=v/v^\sigma$. For fixed
$\sigma$, each $v$ determines at most one element above $\sigma$ that fixes
it. Equivalently, the number of pairs $(g,v)$ with $g$ lying above $\sigma$
and $0\ne v\in W_1$ fixed by $g$ is
$$
 |\{v\ne0:v^{1-\sigma}\in a_\sigma A_L\}|
 \le |A_L|(Q^{1/c}-1)\le |A_L|(\sqrt Q-1),
$$
because the kernel of $v\mapsto v^{1-\sigma}$ is the multiplicative group of
the fixed field $\mathbb{F}_Q^{\langle\sigma\rangle}$, which has order
$Q^{1/c}-1$. For each nonzero $v$, exactly $|L_v|-1$ nonidentity elements
of $L$ fix $v$, and $|L_v|-1\geq s_L-1$. Counting the pairs $(g,v)$ with
$g\in L\setminus\{1\}$, $v\ne0$ and $v^g=v$, first by $v$ and then by the
image $\sigma$ of $g$ in $C_L$, gives
$$
  (Q-1)(s_L-1)\ \le\ \sum_{v\ne0}\bigl(|L_v|-1\bigr)
  \ =\ \sum_{\sigma\ne1}|\{v\ne0:\,v^{1-\sigma}\in a_\sigma A_L\}|
  \ <\ |A_L|\,|C_L|\,\sqrt Q,
$$
whence $|A_L|>(s_L-1)(Q-1)/(|C_L|\sqrt Q)$. Using
$|L|=|A_L||C_L|$ in the preceding bound gives
$z_Ln_L^{3}\ge|L|^{2}/s_L^{6}$. Substituting the inequality for $|A_L|$ and using $(Q-1)^{2}\ge
Q^{2}/2$ for $Q\ge4$, we obtain
$$
  z_Ln_L^{3}\ \ge\ \frac{|A_L|^{2}\,|C_L|^{2}}{s_L^{6}}
  \ >\ \frac{(s_L-1)^{2}(Q-1)^{2}}{Q\,s_L^{6}}
  \ \ge\ \frac{(s_L-1)^{2}}{2\,s_L^{6}}\,Q,
$$
which is at least $1$ as soon as $Q\ge 2s_L^{6}/(s_L-1)^{2}$. Finally,
$$
  s_L\le|C_L|\le |\Gal(\mathbb{F}_Q/\mathbb{F}_p)|
  =\log_p Q\le\log_2 Q.
$$
Hence the sufficient bound
$Q\ge2s_L^{6}/(s_L-1)^{2}$ can fail only when $2^{s_L}\le
Q<2s_L^{6}/(s_L-1)^{2}$. The function
$2s_L^{6}/(s_L-1)^{2}$ is increasing for $s_L\ge2$. At $s_L=18$ the
lower bound $2^{s_L}$ is already at least this function, and the ratio
$2^{s_L}(s_L-1)^2/(2s_L^6)$ increases for $s_L\ge18$. Therefore
$s_L\le17$ and $Q<2\cdot17^{6}/16^{2}<2^{18}$. The finitely many summand
sizes $Q<2^{18}$ that remain are covered, for every possible bottom group $L$, by
Proposition~\ref{prop:semilinear-finite}. Hence $z_Ln_L^{3}\ge1$ for every
chain.
\end{proof}

\begin{theorem}\label{thm:imprimitive}
Let $H$ be a finite solvable group, and let $V$ be a faithful irreducible
imprimitive $H$-module. Then $V$ contains an abundant vector.
\end{theorem}

\begin{proof}
Choose a nontrivial unrefinable system of imprimitivity
$V=W_1\oplus\cdots\oplus W_m$, and let $J$ be the group induced on
$W_1$ by $H_{W_1}$. The group $J$ is solvable and, by
Lemma~\ref{lem:unrefinable-blocks}, nontrivial, faithful, irreducible and
primitive, so Theorem~\ref{thm:local-alternatives} applies. If $J$ has a
balanced pair, Lemma~\ref{lem:branch-balanced} gives an abundant vector in
$V$. If $J$ is metacyclic with no regular vector,
Lemma~\ref{lem:branch-semilinear} gives $z_Ln_L^3\geq1$ for every coordinate
chain. If the action of $J$ on $W_1$ is linearly equivalent to one of the nine residual actions,
the same inequality follows from Theorem~\ref{thm:local-alternatives}. In
either of the last two cases, Lemma~\ref{lem:branch-profile} gives an
abundant vector in $V$.
\end{proof}

\section{Proof of the abundance theorem and of Gluck's conjecture}
\label{sec:gaschutz}

The abundance theorem now follows immediately.

\begin{proof}[Proof of Theorem~\ref{thm:abundance-intro}]
By Corollary~\ref{cor:reduce-irreducible} a counterexample would give a faithful
irreducible $\mathbb{F}_pH$-module that is a counterexample. Such a module is either
primitive or imprimitive, and Theorems~\ref{thm:primitive-action}
and~\ref{thm:imprimitive} rule out both.
\end{proof}

Gluck's conjecture now follows at once. By Theorem~\ref{thm:abundance-intro}
every faithful completely reducible module for a finite solvable group contains an
abundant element. Hence Lemma~\ref{lem:reduction} gives
$|G:\mathbf{F}(G)|\le b(G)^2$ for every finite solvable group $G$. This proves
Theorem~\ref{thm:gluck-holds}.

\appendix

\section{Finite verifications for Gluck's conjecture}
\label{sec:computations}

This appendix describes the four finite calculations used in the proof.
Holt and Yang~\cite{HY23} present a classification of the residual actions
and supply matrix data for the actions in their list. The reconstruction
in~\cite{ZhangRegularOrbits26} finds twelve further
$\GL(8,3)$-conjugacy classes absent from those data and proves that the
corrected classification is complete. Nine of the twelve additional classes
are linearly primitive, and only these are relevant here. We verify only the
abundance and local conditions required for the residual actions, together with
the small permutation and semilinear calculations used earlier in the paper.

\begin{sloppypar}
The package consists of the directories \texttt{code} and \texttt{outputs},
which must remain together. Run each program in Table~\ref{tab:programs} from
the \texttt{code} directory. No program requires the output of another. The
program \path{verify_residual_actions.g} reads
\path{residual_action_data.g}, which contains the
matrix generators and the vectors used below for the $411$ linearly primitive
entries representing matrix groups. Of these, $402$ come from the data
accompanying~\cite{HY23}. The other nine are representatives of the nine
additional linearly primitive classes found in~\cite{ZhangRegularOrbits26}.
Thus the certificates are checked without
repeating either classification calculation. A path beginning with
\path{outputs/} below refers to the package's
\texttt{outputs} directory.
\end{sloppypar}

\begin{table}[H]
\caption{\GAP\ programs and the parts of the proof they verify.}
\label{tab:programs}
\centering
\footnotesize
\renewcommand{\arraystretch}{1.2}
\begin{tabular}{@{}p{0.37\textwidth}|p{0.56\textwidth}@{}}
 program & calculation\\\hline
 \path{verify_permutation.g} &
   Lemmas~\ref{lem:deg5789} and~\ref{lem:weighted-small}(3)\\
 \path{verify_semilinear.g} &
   the six exceptional field orders in Theorem~\ref{thm:semilinear}\\
 \path{verify_residual_actions.g} &
   the certificates for residual actions in Theorem~\ref{thm:residual}\\
 \path{verify_semilinear_threshold.g} &
   the finite calculation in Proposition~\ref{prop:semilinear-finite}\\
\end{tabular}
\end{table}

The four \GAP\ calculations are practical to rerun. In our runs, their CPU
times ranged from a few seconds to less than half an hour. Each used less
than $700$ MB of RAM.

\begin{sloppypar}
The directory \texttt{outputs} contains a transcript of each complete \GAP\
calculation in Table~\ref{tab:programs}. The transcript and program have the
same base name. For example, \path{verify_semilinear.g} is recorded in
\path{outputs/verify_semilinear.out}. Each transcript ends with a summary of
its numerical conclusions. The matrix generators and vectors are also provided
in \path{residual_action_data.m} in \Magma\ form.

The corresponding \Magma\ programs~\cite{BCP97} have the same base names as
the programs in Table~\ref{tab:programs}, with the extension
\path{.magma}. They implement the same checks separately and are not used in
the proof. The \GAP\ and \Magma\ programs for residual actions read the same
stored matrix generators and vectors.
\end{sloppypar}

\subsection{Representation of the groups and the computation of \texorpdfstring{$\nu$}{nu}}
\label{sec:comp-nu}
\begin{sloppypar}
All completed \GAP\ calculations used in the proof were carried out in
\GAP~4.15.1~\cite{GAP4}.
The permutation calculation uses the \GAP\ libraries
\texttt{primgrp}~4.0.1 and \texttt{transgrp}~3.6.5, obtaining groups through
\texttt{PrimitiveGroup} and \texttt{TransitiveGroup}. The programs
\path{verify_semilinear.g} and \path{verify_residual_actions.g} use
\textsf{IRREDSOL}~1.4.4 for irreducibility and primitivity tests. Linear groups are represented
as matrix groups over the prime field $\mathbb{F}_p$. The residual actions are
constructed from the generators in \path{residual_action_data.g}. The semilinear groups
$\GammaL(1,Q)$ are constructed directly over $\mathbb{F}_p$.

To compute $\nu(X)$ exactly, the program sets $\nu(X)=|X|$ when $X$ is
nilpotent. Otherwise it calls \texttt{ConjugacyClassesSubgroups} on $X$ and
takes the greatest order among the representatives found to be nilpotent by
\texttt{IsNilpotentGroup}. Every nilpotent subgroup of $X$ is conjugate to one
of these representatives, so this maximum is $\nu(X)$. The abundance ratio
$\rho_X(x)=|X:X_x|/\nu(X_x)$ is then obtained from the orbit length and
stabiliser, together with $\nu(X_x)$.
\end{sloppypar}

\subsection{Small permutation groups}
\label{sec:comp-perm}
The program \texttt{verify\_permutation.g} finds chain-abundant orbits on
subsets of the required cardinalities for every representative of a solvable
primitive permutation group of degree at most nine in \GAP{}'s library.
For every such representative $J$ on a
domain $D$ of degree at most nine, it forms every $J$-orbit on $\pow(D)$, using
the action on subsets, and every chain $L\trianglelefteq K\trianglelefteq J$,
by letting $K$ range over the normal subgroups of $J$ and $L$ over the normal
subgroups of $K$. There are $21$ such representatives in
\GAP{}'s primitive groups library in degrees $2,\ldots,9$. This is exactly the
family of chains over which
Definition~\ref{def:chain-label} quantifies.
Because $L$ need not be normal in $J$, the value $\rho_L(\Delta)$ can vary as
$\Delta$ ranges over a single $J$-orbit, so examining a single orbit
representative is not sufficient. For every group $L$ occurring as the bottom group of a
chain, the program examines members of the orbit until it finds one satisfying
$|L:L_\Delta|\ge\nu(L_\Delta)$. If none does, every member is examined. The
same bottom group arising from more than one chain is considered only once,
since the inequality does not involve the intermediate group $K$. An orbit is
chain-abundant precisely when a suitable member is found for every such $L$.

The proof needs chain-abundant orbits on subsets of prescribed sizes, not
merely their total number. Of these $21$ solvable primitive permutation
groups, $C_2$ is handled separately. For each of the remaining $20$, the
program verifies chain-abundant orbits on subsets of sizes $1$ and $2$, except
for $\mathrm{A}\Gamma\mathrm{L}(1,9)$ and $\mathrm{AGL}(2,3)$, where it verifies sizes
$2$ and $3$. For those two groups the point stabiliser $J_x$, with $x\in D$,
contains a nilpotent subgroup of order $16>9$, so the singleton orbit is not
abundant for the chain $L=K=J$ and no chain-abundant orbit on singletons exists. For the group $C_2$ of degree two,
the program verifies that there is exactly one chain-abundant orbit, the
orbit of the singleton subsets. The conclusions in degrees $5,7,8,9$ are
precisely those of Lemma~\ref{lem:deg5789}. In degrees three and four the
calculation independently confirms Lemma~\ref{lem:deg34}. The calculation in
degree two confirms that the singleton orbit is the only chain-abundant
orbit, the local fact used in Lemmas~\ref{lem:c2-balance}
and~\ref{lem:c2-weighted}. For
$(a_Z,a_N)=(2,1)$, the condition $a_Z+3a_N\le|D|$ of
Lemma~\ref{lem:weighted-small}(3) becomes $2+3\cdot1\le|D|$. For the two
exceptional groups of degree nine, where $(a_Z,a_N)=(3,2)$, it becomes
$3+3\cdot2\le9$. Thus the cubic pairs of that lemma are verified as well.

The program also verifies directly that all $501$ representatives of solvable
transitive permutation groups in \GAP{}'s \texttt{transgrp} library of
degrees $2,\ldots,15$ have an
abundant subset. This gives a finite check of Theorem~\ref{thm:top-hni} in
those degrees, but is not used in its proof.

\subsection{Small semilinear groups}
\label{sec:comp-semilinear}
\begin{sloppypar}
The program \texttt{verify\_semilinear.g} performs the finite calculation used in
Theorem~\ref{thm:semilinear}. For each of the six field orders
$Q\in\{2,3,4,8,9,16\}$ it constructs $\GammaL(1,Q)$ as a matrix group over
$\mathbb{F}_p$, using multiplication by a primitive element of
$\mathbb{F}_Q$ and the Frobenius map $x\mapsto x^{p}$ as generators relative to
a fixed basis over $\mathbb{F}_p$. It then enumerates the conjugacy classes of subgroups
using \texttt{ConjugacyClassesSubgroups} and retains the solvable subgroups for
which \texttt{IsIrreducibleMatrixGroup} and
\texttt{IsPrimitiveMatrixGroup} both hold. These properties are
invariant under conjugation, so class representatives suffice. For every
retained subgroup $K$, the program computes $n_{\rm reg}(K)$, the number of
$K$-orbits of regular pairs defined in Subsection~\ref{sec:semilinear}, by
decomposing $\mathbb{F}_Q\times\mathbb{F}_Q$ into $K$-orbits and counting
those with trivial stabiliser. It verifies that every proper retained subgroup
$K<\GammaL(1,Q)$ has at least five orbits of regular pairs. For the full
semilinear groups, the values are
$$
  n_{\rm reg}\bigl(\GammaL(1,Q)\bigr)=4,\,4,\,1,\,2,\,3,\,3\qquad
  (Q=2,3,4,8,9,16).
$$
In particular each value is at least $1$, so each group possesses a regular
pair, as Theorem~\ref{thm:semilinear} asserts.
\end{sloppypar}

\subsection{Residual linear groups}
\label{sec:comp-residual}

\begin{theorem}[Residual verification]\label{thm:residual}
Let $H$ be a finite solvable group acting faithfully, irreducibly and
primitively on a finite $\mathbb{F}_p$-space $V$. Suppose that $H$ is not
metacyclic and has no regular orbit. Then both of the following hold.
\begin{enumerate}
\item The module $V$ contains an abundant vector.
\item Either $H$ has a balanced pair, or its action on $V$ is linearly equivalent to one of nine
  residual actions that are transitive on $V\setminus\{0\}$ and satisfy the
  cubic condition $z_Ln_L^{3}\ge1$ for every chain
  $L\trianglelefteq K\trianglelefteq H$.
\end{enumerate}
\end{theorem}

\begin{proof}
By Theorem~\ref{thm:residual-classification}, every action satisfying the
hypotheses is linearly equivalent to one of the actions represented by the
$411$ linearly primitive entries described at the beginning of this appendix.
It is therefore enough to verify the action represented by each entry. These
$411$ entries represent $353$
conjugacy classes in their respective general linear groups, so repetitions
merely repeat some of the
verifications. The program
\path{verify_residual_actions.g} checks the matrix degree, ground field,
solvability, irreducibility and primitivity before checking the certificate
for each entry.

\subsubsection*{Abundant vectors}
Fix one of the $411$ entries, let $V=\mathbb{F}_p^{\,d}$, and let
$J\leq\GL(V)$ be the matrix group represented by that entry. We exhibit a
vector $v$ with $\rho_J(v)=|J:J_v|/\nu(J_v)\ge1$ and verify this inequality
directly. The calculation uses one of the following three criteria:
\begin{enumerate}
\item[(R)] \emph{regular:} if $|J_v|=1$ then $\rho_J(v)=|J|\ge1$.
\item[(O)] \emph{order bound:} if $|J:J_v|\ge|J_v|$ then, since $\nu(J_v)\le|J_v|$,
  $$
    \rho_J(v)\ge\frac{|J:J_v|}{|J_v|}\ge1.
  $$
  This uses only the orbit length and the stabiliser order, and does not
  require $\nu(J_v)$ to be computed.
\item[(E)] \emph{exact calculation of $\nu$:} otherwise compute $\nu(J_v)$ by the method of
  Appendix~\ref{sec:comp-nu} and verify $|J:J_v|\ge\nu(J_v)$ directly.
\end{enumerate}
The program constructs distinct points of the orbit $v^J$ until the complete
orbit has been obtained or until the following bound proves criterion~(O). If
$\lfloor\sqrt{|J|}\rfloor+2$ distinct points have been found, the orbit has
length greater than $\sqrt{|J|}$, so
$|J:J_v|>\sqrt{|J|}>|J_v|\ge\nu(J_v)$
and criterion~(O) already applies. The calculation therefore stops at that point.
This bound makes the computation
feasible: several residual modules are large --- already
$|V|=3^{16}=43\,046\,721$ occurs --- so a full orbit enumeration is
impractical, whereas this calculation constructs only slightly more than
$\sqrt{|J|}$ orbit points. Only if the whole orbit has been obtained before
this bound is reached does the program compute the exact stabiliser and, where
necessary, $\nu(J_v)$. For the ten groups requiring criterion~(E), the resulting
stabilisers are small enough for their conjugacy classes of subgroups to be
computed.

The file \path{residual_action_data.g} lists the vectors explicitly. They are
not searched for during the verification. The file contains one certificate for
each of the $411$ linearly primitive entries. Matrix and
vector entries in $\mathbb{F}_p$ are written as integers from $0$ to $p-1$.
The file contains no orbit lengths, stabilisers or values of
$\nu$. The program computes the information required by the applicable
criterion.
By the classification, none of the $411$ listed actions has a
regular orbit, so criterion~(R) does not occur. For $401$ entries the listed vector is
verified by~(O), and for the remaining $10$ it is verified by~(E). This proves
item~(1).

\subsubsection*{Balanced pairs and cubic conditions}
Item~(2) is used in Theorem~\ref{thm:local-alternatives}. In the case of a regular pair, the vectors
$a,b\in V$ satisfy $J_a\cap J_b=1$ and $b\notin a^{J}$.
Thus $(a,b)$ is a regular pair with entries in distinct orbits, and
$(a^J,b^J)$ is balanced by Lemma~\ref{lem:balanced-sources}(3). The program
computes $J_a\cap J_b$ as the stabiliser of
$b$ in the previously computed stabiliser $J_a$, and verifies $b\notin a^J$
by testing membership in the explicitly
enumerated orbit $a^J$.
\begin{sloppypar}
The program uses the two vectors listed in
\path{residual_action_data.g} for $394$ entries. For eight
further entries, all on modules of order at most $81$, the same file lists
vectors in two distinct nonzero orbits. The program
enumerates those two orbits and, for each group $L$ occurring as the bottom
group of a chain $L\trianglelefteq K\trianglelefteq J$, verifies that each
orbit contains a vector $a$ with $\rho_L(a)\geq1$. The two orbits are therefore
chain-abundant and form a balanced pair by
Lemma~\ref{lem:balanced-sources}(1). Finally, the groups represented by the
remaining nine entries are transitive on the nonzero vectors of their modules, so
$\{0\}$ and $V\setminus\{0\}$ are their only orbits. For each of them the
program enumerates every nonzero vector and verifies the \emph{cubic condition}
$$
  z_Ln_L^{3}=\frac{\bigl(\max_{v\ne0}\rho_L(v)\bigr)^{3}}{\nu(L)}\ \ge\ 1
  \qquad\text{for every chain } L\trianglelefteq K\trianglelefteq J,
$$
which is the hypothesis of Lemma~\ref{lem:branch-profile}. These nine actions
are listed in
Table~\ref{tab:cubic-actions}. Its first column gives the row, type
and entry number in the corresponding Holt--Yang data file~\cite{HY23}.
\end{sloppypar}
\begin{table}[H]
\caption{Residual actions for which the cubic condition is used.}
\label{tab:cubic-actions}
\centering
$
\renewcommand{\arraystretch}{1.15}
\begin{array}{c|c|c}
 \text{Holt--Yang row and entry} & V & |J|\\\hline
 19\,(E^{-}):2 & \mathbb{F}_3^{\,4} & 640\\
 19\,(E^{-}):5 & \mathbb{F}_3^{\,4} & 320\\
 19\,(E^{-}):6 & \mathbb{F}_3^{\,4} & 160\\
 62:1 & \mathbb{F}_3^{\,2} & 48\\
 62:2 & \mathbb{F}_3^{\,2} & 24\\
 63:1 & \mathbb{F}_5^{\,2} & 96\\
 63:2 & \mathbb{F}_5^{\,2} & 48\\
 64:1 & \mathbb{F}_7^{\,2} & 144\\
 66:1 & \mathbb{F}_{11}^{\,2} & 240
\end{array}
$
\end{table}
Of the nine, entry $62:2$ can be checked by hand. Here
$J\cong\SL(2,3)$ acts on $V=\mathbb{F}_3^{\,2}$ and
$\nu(J)=8$, attained by the quaternion subgroup $Q_8$. The group $J$ is
transitive on the eight nonzero vectors with point stabiliser of order $3$,
so every nonzero $v$ has $\rho_J(v)=(24/3)/\nu(C_3)=8/3$, whence $n_J=8/3$
and $z_Jn_J^{3}=(1/8)(8/3)^{3}=64/27\ge1$. For the remaining possible bottom
groups, the calculation is immediate:
$Q_8$ acts regularly on the eight nonzero vectors, so $z_Ln_L^{3}=8^{3}/8=64$ for
$L=Q_8$, and its subgroups have trivial point stabilisers, giving $16$ for $L=C_4$, $4$ for
$L=C_2$ and $1$ for $L=1$. Every chain $L\trianglelefteq K\trianglelefteq J$
therefore satisfies the cubic condition, with the minimum value $1$ attained
when $L=1$.
When checking chain-abundant orbits and the cubic condition, the program
examines every subgroup $L$ that can occur as the bottom group of a chain
$L\trianglelefteq K\trianglelefteq J$. Since the inequalities do not involve
the intermediate group $K$, this covers every coordinate chain that can
arise in Theorem~\ref{thm:coordinate}. No
separate enumeration of chains is needed for regular pairs, because
Lemma~\ref{lem:balanced-sources}\textup{(3)} applies to every chain. The program
verifies regular pairs for $394$ entries, chain-abundant orbits for eight
entries, and the cubic condition for nine entries. These numbers sum to
$411$, proving~(2).
\end{proof}

\subsection{The semilinear threshold}
\label{sec:comp-threshold}

\begin{proposition}[Semilinear subgroups below the threshold]\label{prop:semilinear-finite}
Let $Q<2^{18}$ be a prime power and let $L\le\GammaL(1,Q)$ act on $\mathbb{F}_Q$.
Set $z_L=1/\nu(L)$ and $n_L=\max_{v\ne0}\rho_L(v)$. Then $z_Ln_L^{3}\ge1$.
\end{proposition}

\begin{proof}
Write $Q=p^f$. If $f=1$, then
$L\leq\GammaL(1,p)=\mathbb{F}_p^\times$ is cyclic and every nonzero
vector has trivial stabiliser. Hence $z_Ln_L^3=|L|^2\geq1$.
Assume that $f\geq2$. The program
\texttt{verify\_semilinear\_threshold.g} treats all $149$ prime powers
$Q=p^f<2^{18}$ of this form.

Choose a generator $\zeta$ of $\mathbb{F}_Q^\times$ and identify
$\zeta^t$ with $t$ modulo $Q-1$. With actions on the right,
$$
 \GammaL(1,Q)=
 \langle s,\varphi\mid s^{Q-1}=\varphi^f=1,\ s^\varphi=s^p\rangle,
$$
where $s$ sends $t$ to $t+1$ and $\varphi$ sends $t$ to $pt$.
Every subgroup has a unique description
$$
 L(u,r,j)=\langle s^u,\varphi^rs^j\rangle,
$$
where $u\mid Q-1$, $r\mid f$, $0\leq j<u$, and
$u\mid jT_{r,c}$, with
$c=f/r$ and
$T_{r,c}=1+p^r+\cdots+p^{(c-1)r}$.
Indeed, $\langle s^u\rangle=L\cap\langle s\rangle$, the image of $L$
in the Galois quotient is generated by the image of $\varphi^r$, and
$(\varphi^rs^j)^c=s^{jT_{r,c}}$.
Thus the divisibility condition is necessary and sufficient, and
$$
 |L(u,r,j)|=\frac{Q-1}{u}\,c.
$$
The program runs through all parameters satisfying these conditions
and therefore includes every subgroup.

The $\langle s^u\rangle$-orbits on $\mathbb{F}_Q^\times$ correspond to
the residue classes modulo $u$, each of length $(Q-1)/u$.
The element $\varphi^rs^j$ acts on these classes as
$$
 \alpha(t)=p^rt+j\pmod u.
$$
For each $h\mid c$, the fixed points of $\alpha^h$ are the solutions of
$$
 (p^{rh}-1)t\equiv
 -j(1+p^r+\cdots+p^{(h-1)r})\pmod u.
$$
This congruence has either no solution or
$\gcd(p^{rh}-1,u)$ solutions. Subtracting the points of smaller least
period determines the greatest cycle length $\ell_L$ of $\alpha$.
Hence the greatest $L$-orbit length is
$$
 m_L=\frac{Q-1}{u}\ell_L.
$$
If $|v^L|=m$, then $L_v$ embeds in the cyclic Galois quotient. Thus
$\nu(L_v)=|L_v|=|L|/m$, and consequently
$$
 \rho_L(v)=\frac{m^2}{|L|}
 \qquad\text{and}\qquad
 n_L=\frac{m_L^2}{|L|}.
$$

Since $\nu(L)\leq|L|$, the required inequality follows whenever
$n_L^3\geq|L|$. The program finds only three nontrivial subgroups for
which $n_L^3<|L|$. For these groups it computes $\nu(L)$ from an
isomorphic polycyclic presentation:
$$
\begin{array}{c|c|c|c|c|c}
 Q&(u,r,j)&|L|&m_L&\nu(L)&n_L^3/\nu(L)\\\hline
 4 &(1,1,0)&6 &3 &3 &9/8\\
 8 &(1,1,0)&21&7 &7 &49/27\\
 16&(1,1,0)&60&15&15&225/64
\end{array}
$$
All three ratios exceed $1$. The program finds one nontrivial equality case,
with $Q=9$, $(u,r,j)=(2,1,0)$, $|L|=\nu(L)=8$, $m_L=4$ and $n_L=2$.
Altogether the program enumerates $1\,896\,985$ nontrivial subgroups. The
trivial subgroup also gives equality.\qedhere
\end{proof}

\section*{Acknowledgements}

I thank my supervisors, Prof. David Craven and Prof. Chris Parker, for reading this paper and providing helpful comments. I am grateful to
Dr. Luca Sabatini for introducing me to Gluck's conjecture on an afternoon in November 2024 at the Mathematics Institute, University of Warwick. I also thank Prof.~Derek Holt for a helpful discussion concerning possible corrections to the classification in~\cite{HY23} during the conference \emph{Modern Flavours of Finite Groups}. During the preparation of this paper, the GPT-5.4 model was used to assist with debugging and improving some verification programs and translating \GAP~code into \Magma. I checked and revised the resulting text and code, verified the computations, and take full responsibility for the contents of the paper.

\end{document}